\documentclass[fontsize=12pt,a4paper,headings=normal,
twoside=false,leqno,parskip=half-,abstract=true]{scrartcl}
\usepackage[english]{babel}
\usepackage[utf8]{inputenc}
\usepackage{hyperref}
\hypersetup{
 pdftitle={Global Hopf Networks},
 pdfauthor={Bernold Fiedler},
 colorlinks=true,
 linkcolor=blue,
 citecolor=blue,
 filecolor=blue,
 urlcolor=blue}
 
 \usepackage{afterpage}

\usepackage{graphicx}
\usepackage[format=plain,labelfont=bf,font=small]{caption}
\usepackage{xcolor}
\usepackage[arrow, matrix, curve]{xy}
\usepackage{float}

\usepackage{caption}
\usepackage{enumitem}  
\usepackage{tabulary}
\usepackage{array}
\newcolumntype{N}[1]{>{\centering\arraybackslash}m{#1}}

\usepackage{amsmath,amsthm}
\swapnumbers 
\usepackage{amssymb} 

\makeatletter
\newcommand{\tpitchfork}{%
  \vbox{
    \baselineskip\z@skip
    \lineskip-.52ex
    \lineskiplimit\maxdimen
    \m@th
    \ialign{##\crcr\hidewidth\smash{$-$}\hidewidth\crcr$\pitchfork$\crcr}
  }%
}
\makeatother
\usepackage{latexsym}
\newcommand{\Mod}[1]{(\mathrm{mod}\ #1)}

\usepackage[notref,notcite,color,final 
]{showkeys}

\definecolor{refkey}{rgb}{1,0,0}
\definecolor{labelkey}{rgb}{1,0,0}

\usepackage{tikz}

\def\zhong{\leavevmode\vbox{\offinterlineskip\halign{\tabskip=0pt\hfil##\hfil\cr\vrule height 2ex \cr\noalign{\vskip  -1.3ex}$\Box$\cr}}}

\usepackage[textwidth=2cm,textsize=small,backgroundcolor=none]{todonotes}

  \mathchardef\ordinarycolon\mathcode`\:
  \mathcode`\:=\string"8000
  \begingroup \catcode`\:=\active
    \gdef:{\mathrel{\mathop\ordinarycolon}}
  \endgroup

\theoremstyle{plain}
\newtheorem{thm}{Theorem}[section]
\newtheorem{lem}[thm]{Lemma}
\newtheorem{prop}[thm]{Proposition}
\newtheorem{cor}[thm]{Corollary}
\newtheorem{defi}[thm]{Definition}

\makeatletter

\@addtoreset{equation}{section}
\makeatother

\begin{document}

\title{\LARGE{Global Hopf bifurcation in networks\\
with fast feedback cycles}}

\author{
 \\
Bernold Fiedler*\\
{~}\\
\emph{Dedicated to Alexander Mielke} \\
\emph{on the occasion of his sixtieth birthday}\\
\vspace{2cm}}

\date{version of \today}
\maketitle
\thispagestyle{empty}

\vfill

*\\
Institut für Mathematik\\
Freie Universität Berlin\\
Arnimallee 3\\ 
14195 Berlin, Germany\\
\\


\newpage
\pagestyle{plain}
\pagenumbering{roman}
\setcounter{page}{1}

\begin{abstract}

\noindent
Autonomous sustained oscillations are ubiquitous in living and nonliving systems. As open systems, far from thermodynamic equilibrium, they defy entropic laws which mandate convergence to stationarity.
We present structural conditions on network cycles which support global Hopf bifurcation, i.e. global bifurcation of non-stationary time-periodic solutions from stationary solutions.
Specifically, we show how monotone feedback cycles of the linearization at stationary solutions give rise to global Hopf bifurcation, for sufficiently dominant coefficients along the cycle.

\medskip\noindent
We include four example networks which feature such strong feedback cycles of length three and larger:
Oregonator chemical reaction networks, 
Lotka-Volterra ecological population dynamics,
citric acid cycles, and 
a circadian gene regulatory network in mammals.
Reaction kinetics in our approach are not limited to mass action or Michaelis-Menten type.

\end{abstract}

\vspace{2cm}
\tableofcontents


\newpage
\pagenumbering{arabic}
\setcounter{page}{1}

\section{Reaction networks and oscillations}\label{sec1a}

Network graphs are a common modeling device to describe dependencies of certain sub-units among each other.
Vertices indicate those sub-units.
Directed edges indicate coupling directions, or positive and negative signs of influence.
Popular examples in a differential equations context are chemical reaction systems, neural networks, power grids, and many others.
Where emphasis may have been on equilibration and steady state behavior, originally, more recent focus has shifted much towards the complexities of temporal and spatial patterns of the collective vertex behavior.
The main objective, in the present paper, is to explore the potential of network structures, as such, towards autonomously time periodic network responses.
Particulars of coupling parameters will play a subordinate role in that quest.
Mostly we address large ranges of parameters.
A slow-fast constraint, however, will emphasize a select feedback cycle in the network.

Let us be more specific.
Chemical reaction networks, for example, take the form 
\begin{equation}
\dot{x}= f(x)= \sum_j(\bar{y}_j-y_j)r_j(x)
\label{eq:1.1}
\end{equation}
with positive vectors $x$ of the metabolite concentrations $x_m,\ m=1, ..., M,$ finitely many nonnegative \emph{stoichiometric coefficient} vectors $y_j\neq \bar{y}_j\in \mathbb{R}^M$, and positive reaction rate functions
\begin{equation}
r_j>0\,.
\label{eq:1.2}
\end{equation}

In chemical notation, the $j$-th summand in \eqref{eq:1.1} accounts for the reaction
\begin{equation}
j: \quad  y_{j1}X_1+\ldots +y_{jM}X_M \ \longrightarrow \ \bar{y}_{j1} X_1+\ldots +\bar{y}_{jM}X_M\,.
\label{eq:1.3}
\end{equation}
One possibility to view \eqref{eq:1.1} as a network takes the metabolites $X_m$ with concentrations $x_m$, as vertices $m$, and dependencies of $f_{m}$ on $x_{m'}$, as directed edges $m' \rightarrow m$.
See our comments on the general setting \eqref{eq:1.10} below.
Another possibility, suggested by \eqref{eq:1.3}, is to take the vectors $y,\bar{y}$ as vertices, with reaction arrows as edges. 
See \eqref{eq:1.9a} below.

\emph{Educts} or \emph{inputs} $m$ of reaction $j$ are defined by nonzero $y_{jm}>0$, and \emph{outputs} by nonzero $\bar{y}_{jm}$.
Nonzero $\bar{y}_{jm}=y_{jm}$ describe a \emph{catalyst} $m$, for which reaction $j$ does not affect $x_m$.
\emph{Strong autocatalysis} of $m$, which catalyzes its own net production, is described by
\begin{equation}
\bar{y}_{jm}>y_{jm}\,.
\label{eq:1.4}
\end{equation}

\emph{Mass action kinetics}, prevalent in large parts of classical anorganic chemistry, and in gas phase reactions in particular, is defined by
\begin{equation}
r_j(x)=k_jx^{y_j}:=k_{j}x_{1}^{y_{j1}}\cdot\ldots\cdot x_M^{y_{jM}}\,,
\label{eq:1.5}
\end{equation}
usually for integer-valued $y_j, \bar{y}_j$, with the convention $x_m^0:=1$.
The \emph{rate coefficients} $k_j$ are assumed to be strictly positive; see \eqref{eq:1.2}. Reactions catalyzed by enzymes, ubiquitous in biological metabolic networks, allow for more general \emph{Michaelis-Menten kinetics} of the form
\begin{equation}
r_j(x)=k_j \prod_m\, (x_m/(1+c_{jm}x_m))^{y_{jm}}
\label{eq:1.6}
\end{equation}
with saturation coefficients $c_{jm}\geqslant0$.
Usually $y_{jm}\in\{0,1\}$.
Note how \eqref{eq:1.6} reduces to mass action \eqref{eq:1.5}, for $c_{jm}=0$.
The denominator is often ``linearized'', inappropriately, to become $1+c^T_jx$.
The partial derivatives of \eqref{eq:1.6} satisfy
\begin{equation}
r_{jm}:= \partial_{x_m}r_j>0 \qquad \Longleftrightarrow \qquad y_{jm}>0.
\label{eq:1.7a}
\end{equation}
\emph{Enzymatic inhibition} of $r_j$ by $x_m$, in contrast, is characterized by $\bar{y}_{jm}=y_{jm}$ and a factor $1/(1+c_{jm}x_m)$, so that $r_{jm}<0$ instead.
See sections \ref{sec5.3} and \ref{sec5.4} for examples.
Therefore we will not subscribe to any monotonicity constraints, in the general mathematical setting of section \ref{sec1b} below.
For in-depth information on the very rich subject of chemical reaction kinetics we refer to the currently 43 volumes of the book series Comprehensive Chemical Kinetics \cite{CCK}.
For a comprehensive background on chemical reaction networks see \cite{Fei19}, by a leading pioneer in the field.

Thermodynamics of closed systems advocates convergence to \emph{steady state equilibria} $x^*$ of \eqref{eq:1.1}, i.e. to stationary solutions 
\begin{equation}
0= f(x^*)=\sum_j \, (\bar{y}_j-y_j) r_j(x^*),
\label{eq:1.7b}
\end{equation}
The basic tool is a Lyapunov function
\begin{equation}
V(x):= \sum_m \,   f_m(x)\cdot v(x_m/x_m^*),
\label{eq:1.8}
\end{equation}
with $v(\xi):=\xi \log\xi-\xi+1$, related to the negative of \emph{relative entropy}.
The defining property of Lyapunov functions is that $t\mapsto V(x(t))$ is decreasing along solutions $x(t)$, usually strictly when $x(t)$ is nonstationary.
The Lyapunov property of $V$ has been established in \cite{HJ74} under the assumptions of mass action kinetics \eqref{eq:1.5} and the following complex balance condition \eqref{eq:1.9a}.

From \eqref{eq:1.3} we recall how the stoichiometric vectors $y_j, \bar{y}_j$ may be taken as vertices of the \emph{complex graph} $\mathcal{C}$, possibly including the complex $y_j=0$ and/or $\bar{y}_j=0$.
Note how identical vectors $y_j$ or $\bar{y}_j$ for different $j$ may describe the same vertex complex.
The directed edges $j$ of $\mathcal{C}$ are simply the reaction arrows \eqref{eq:1.3} of standard chemical notation.
Then \emph{complex balance} requires the existence of a positive equilibrium $x^*>0$ such that the inflows and outflows balance, i.e.
\begin{equation}
\sum_{j:\ \bar{y}_j=y} r_j(x^*)= \sum_{j':\ y_{j'}=y} r_{j'}(x^*)\,,
\label{eq:1.9a}
\end{equation}
at every nonzero complex $y$.
In other words, the total production rate at any complex $y=\bar{y}_j$ as an output of reactions $j$, balances the total consumption rate at the same complex $y=y_{j'}\,$, as an input of other reactions $j'\neq j$, just as currents do at vertices of Kirchhoff circuits.

\emph{Detailed balance} in reversible reaction systems $j^\pm$: $y_{j^\pm}\leftrightharpoons \bar{y}_{j^\pm}$, where $\bar{y}_{j^+}=y_{j^-}$ and $\bar{y}_{j^-}=y_{j^+}$ for all $j^\pm$, is a special case of complex balance, already considered by \cite{Weg1902}. It requires $r_{j^+}(x^*)=r_{j^-}(x^*)$ for every reversible pair $j^\pm$.
For reversible monomolecular cycles $y_{j^+}=\mathbf{e}_j, \ y_{j^-}=\mathbf{e}_{j+1}\,,\ j\,\Mod N$ and mass action kinetics, detailed balance amounts to the famous Wegscheider condition
\begin{equation}
\prod^N_{j=1}   k^+_j=\prod^N_{j=1} k^-_j,
\label{eq:1.9b}
\end{equation}
which prevents oscillations.
Wegscheider's arguments for \eqref{eq:1.9b} were based on thermodynamic considerations on irreversibility, at the microscopic level.
It is a lasting merit of \cite{Hir1911} to point at the possibility of (damped) oscillations, once the Wegscheider constraints \eqref{eq:1.9b} are strongly violated.
Our emphasis below on unidirectional $N$-cycles, as a cause for global Hopf bifurcation, is essentially based on this insight.

In passing we note how reversible monomolecular cycles lead to \emph{Jacobi systems}
\begin{equation}
\dot{x}_m=f_m(x_{m-1}\,, x_m, x_{m+1}),
\label{eq:1.9c}
\end{equation}
for $m\,\Mod N$ with strictly positive off-diagonal partial derivatives $\partial_{x_{m\pm1}} f_m$.
See \cite{FuOl88} for a detailed study.
Standard mass action makes 
\begin{equation}
f_m=k^+_{m-1} x_{m-1}-(k^-_{m-1}+k^+_m)x_m + k^-_m x_{m+1}
\label{eq:1.9d}
\end{equation}
linear, and $x_1+\ldots+x_N\equiv \mathrm{const} $ is preserved.
Spectral analysis, similar to the case $\beta=+1$ in proposition \ref{prop:3.1} below, then implies stability of steady states, due to the presence of a positive (left) kernel vector.
Alternatively, complex balance for the unit vector complexes $y_m=\mathbf{e}_m$ can be invoked.
Note how the addition of strongly autocatalytic diagonal terms like $X_m\rightarrow2X_m$ can lead to sustained oscillations and instability.
Similar remarks apply to $N$-cycles with general monotone reaction rates $r^\pm_j$.

Complex balance is clearly sufficient for $x^*$ to be a stable steady state \eqref{eq:1.7b} of \eqref{eq:1.1}.
Notably \cite{Mie17} has much extended the ODE stability results, for the mass action Lyapunov function $V$ in \eqref{eq:1.8}, to a reaction-diffusion PDE context under Neumann boundary conditions.
His results start from a general observation in \cite{Ali79}. 
Mielke's extensions include exponential convergence results and cover the presence of stoichiometric invariant subspaces.

Our present paper  will remain in an ODE setting, for simplicity of presentation, even though our approach extends to the PDE setting of reaction-diffusion systems.  
Based on fast $N$-cycles, we study the appearance of time periodic solutions, in contrast to equilibration and beyond the variational complex balance setting \eqref{eq:1.9a}.

Experimental evidence for chemical oscillations has become overwhelming, by now \cite{Zha91}.
We recall a few highlights.
As early as 1828, Fechner has observed transient polarity reversals in an electro-chemical experiment \cite{Fe1828}; see also \cite{He1901}.
The celebrated integrable Lotka-Volterra model \cite{Lot1920} has been described by Lotka, originally, as a hypothetical model for sustained time-periodicity in a chemical reaction with mass action kinetics, and not in the tradited Volterra context \cite{Vol1931} of predator-prey population dynamics.
A first chemical experiment with sustained autonomous oscillations was described in \cite{Bray1921}.
Experiments on the now famous \emph{Belousov-Zhabotinsky reaction} (BZ) by Belousov in the 1950s were rejected, at first, on ``obvious" thermodynamic grounds.
The BZ reaction is an open system, in fact, and the eventual process of dieing out is irrelevant over the time-span of its oscillations.
A decade later, young Zhabotinsky rehabilitated the findings by Belousov, and managed to get published \cite{Zha64}.
The famous \emph{Brusselator} ``model" \cite{Lef68, PriLef68} for the BZ reaction, by Prigogine and co-workers, had originally just been designed to exhibit and numerically investigate Turing instability \cite{Tur52}.
A model for observed \emph{glycolytic oscillations} in the metabolism of yeast cells was suggested by \cite{Sel68}.
All the above considerations were based on phase plane analysis, i.e. on reaction systems \eqref{eq:1.1} with $M=2$ metabolites $m=1,2$.
The article \cite{Hig67} provides a comprehensive survey and discussion of the planar possibilities.

The chemically more realistic \emph{Oregonator} model \cite{FN74} of the BZ-reaction was an early example of oscillatory chemical reaction networks involving at least $M=3$ metabolites; see section \ref{sec5.1}.
Eigen's quite hypothetical \emph{hypercycle} \cite{Eig71}, of course, also known as the \emph{replicator equation}, features cycles of any length $N$ in an attempt to model molecular evolution; see also the book \cite{HS98}.
It can be seen as a projective version, for population percentages, of general \emph{Lotka-Volterra models} \cite{Oli14} discussed in section \ref{sec5.2}.
Oscillations in the famous \emph{citric acid cycle} (CAC, Krebs cycle) involving eight metabolites have been described, experimentally, by \cite{MacDetal03}; see section \ref{sec5.3} below.
In section \ref{sec5.4} we discuss a gene regulatory network for \emph{circadian rhythms} in mammals \cite{Miretal09}.
Non-isothermal oscillations, where the temperature dependence of the rate functions $r_j$ plays a decisive role, have been studied much, in the PDE context of spatially heterogeneous catalysis. See \cite{Aris75, Fie83} for experimental  and mathematical results, as well as \cite{IE95} for a survey of the early developments.

Theoretical results on autonomous time-periodic oscillations are rare, beyond mere numerical simulation.
Mostly, they establish the existence of stationary solutions $f(x^*)= 0$ with purely imaginary  eigenvalues, by the $M$-dimensional Routh-Hurwitz criterion.
Classical local Hopf bifurcation \cite{Hopf1942, MaMcC76, CR77} then is supposed to infer periodic solutions.
For mass action kinetics \eqref{eq:1.5}, however, the computational difficulties seem to grow prohibitively with dimension.
Even best analytic results like \cite{GES05, EEetal15}, which require advanced techniques and concepts from computational algebra, do not proceed beyond $M=3,4$. 
They also fail to address standard prerequisites of local Hopf bifurcation, like spectral nonresonance and transverse crossing conditions.
Instead, our approach will avoid the restrictions of mass action kinetics, and will explore fast feedback cycles in networks as a source of global Hopf bifurcation.

In the next section we outline our main result on global Hopf bifurcation of periodic orbits induced by fast feedback cycles, in a much more general mathematical setting. 
See our main result, theorem \ref{thm:1.2}.
Our precise notion of global Hopf bifurcation involves some subtleties which we postpone to section \ref{sec2}; see definition \ref{def:2.2} and corollary \ref{cor:2.4}.
We recall some tools for global Hopf bifurcation there, as developed in \cite{Fie85}; see theorem \ref{thm:2.3} and corollary \ref{cor:2.4}.
In section \ref{sec3} we collect the prerequisite spectral properties of cyclic monotone feedback systems, in the spirit of \cite{M-PS90}.
In section \ref{sec4}, this enables us to prove theorem \ref{thm:1.2} as an application of corollary \ref{cor:2.4}.
We conclude with the promised four applications, in section \ref{sec5}.

\textbf{Acknowledgment.} 
The present results are based entirely on many pleasantly challenging and always inspiring discussions with Alexander Mielke, including  his beautiful contribution \cite{Mie17}.
The present paper is dedicated to him in deep gratitude and lasting friendship.
We are also indebted to Marty Feinberg, for sharing his pioneering insights so generously and encouragingly, over so many years, and to Nicola Vassena for many helpful comments concerning network dynamics.
Patricia H\u{a}b\u{a}\c{s}escu performed the typesetting with outstanding dedication and diligence, and some expert help from Alejandro Lopez.
This work was partially supported by the Deutsche Forschungs\-gemeinschaft through SFB 910 project A4.

\section{Main result}\label{sec1b}
In this section we formulate our main result, theorem \ref{thm:1.2}, for general ODEs
\begin{equation}
\dot{x}_m=f_m(x),
\label{eq:1.10}
\end{equation}
$m=1, \ldots, M,$ with $C^1$-nonlinearities $f_m$ and for $x\in \mathbb{R}^M$.
Let $x^*$ be a stationary solution of \eqref{eq:1.10}, i.e.~$f(x^*)=0$.
Then the Jacobian $f_x(x^*)=(f_{mm'})$ is given by the partial derivatives
\begin{equation}
f_{mm'}:= \partial_{x_{m'}} f_m(x^*),
\label{eq:1.11}
\end{equation}
for $1\leqslant m, m'\leqslant M$.

In network language, the ODE setting \eqref{eq:1.10} generalizes \eqref{eq:1.1} as follows.
Suppose $f_m=f_m(x_{I(m)})$\,. 
This notation indicates that $f_m$ depends on the component $x_{m'}$ of $x\in\mathbb{R}^M$ if, and only if, $m'$ is in the set $I(m)\subseteq\{1, \ldots, M\}$ of \emph{inputs} of $f_m$.
Then the dependencies $I(m)$ of $f_m$ define a directed graph $\Gamma$ with metabolite vertices $m$ and directed edges $m'\rightarrow m$ from $m'\in I(m)$ to $m$.
We explicitly allow, but do not impose, self-loops $m\in I(m)$.
The graph $\Gamma$ in the setting \eqref{eq:1.10} is commonly employed in the description of gene regulatory networks; see for example \cite{FieMKS13}.
Note $f_{mm'}=0$, unless $m'\in I(m)$.
The full ODE \eqref{eq:1.10} corresponds to the maximal bi-directional graph $\Gamma$ with self-loops, of course.

Based on nonzero entries $f_{mm'}\,,\ m'\in I(m)$, of the Jacobian $f_x(x^*)$, we can now identify fast feedback cycles in the directed graph $\Gamma$, which play the central role for our results on fast oscillations.
For the sake of generality, however, we do not restrict ourselves to any specific directed graph $\Gamma$.
Rather, we focus on dominant $N$-cycles, just on the linear level of the Jacobian $f_x(x^*)$, as a source of oscillations and of global Hopf bifurcation.

\begin{defi}\label{def:1.1}
Fix $N\in\{2,\ldots,M\}$.
Let $\mathbf{m}= (m_1\ \ldots\ m_N)$ denote any ordered $N$-tuple of distinct metabolites $m_k\in \{1, \ldots, M\}$.
We call $\mathbf{m}$ an $N$-\emph{cycle} if
\begin{equation}
\beta_k := f_{m_km_{k-1}} \neq 0
\label{eq:1.12}
\end{equation}
holds, for all indices $k\,\Mod N$.
We say the $N$-cycle possesses \emph{positive} or \emph{negative} feedback, depending on the sign of 
\begin{equation}
\beta := \prod_{k=1}^N \beta_k \neq 0.
\label{eq:1.13}
\end{equation}
For a \emph{nondegenerate} $N$-\emph{cycle} we require, in addition, nonzero self-loops
\begin{equation}
a_k := -f_{m_km_k} \neq 0,
\label{eq:1.14}
\end{equation}
for all $k=1, \ldots, N$.
Motivated by reaction network dynamics, we call the number $0\leq N_\mathrm{aut}\leq N$ of $a_k<0$, i.e. the number of strictly positive self-feedbacks $f_{m_km_k}$, the \emph{autocatalytic number} of the nondegenerate $N$-cycle $\mathbf{m}$.
\end{defi}

We can now describe the detailed setting of our main result, theorem \ref{thm:1.2} below.
It is of crucial importance here, and deviates significantly from previous work in the area, that \emph{we consider the partials}  $f_{mm'}$ \emph{as free parameters which may vary independently of the steady state} $x^*$ \emph{and, to some extent, independently of each other.} 
More precisely we consider networks $\dot{x}=f(\varepsilon, a, x)$ depending on a parameter $a>0$ and a small parameter $\varepsilon > 0$, such that
\begin{equation}
0= f(\varepsilon, a, x^*)
\label{eq:1.15}
\end{equation}
possesses a parameter-independent stationary solution $x^*$.

For the Jacobian at $x^*$ we assume an expansion
\begin{equation}
f_x(\varepsilon, a, x^*)=
\bigg(
\begin{array}{cc}
\mathbf{A}+ \varepsilon\mathbf{A}' &  \varepsilon \mathbf{B}\\
\varepsilon \mathbf{C} & \varepsilon \mathbf{D}
\end{array}\bigg ),
\label{eq:1.16}
\end{equation}
in block matrix form, with small $\varepsilon >0$.
Only $\mathbf{A}=\mathbf{A}(a)$ is allowed to depend on the parameter $a$.
Specifically, we assume that the $N \times N$ block matrix $\mathbf{A} = (f_{mm'})_{1\leq m, m'\leq N}$ of $f_x$, at $\varepsilon = 0$, describes a nondegenerate $N$-cycle, $\det \mathbf{A} \neq 0$.
Here we have relabeled $\mathbf{m}$, without loss of generality, such that $\mathbf{m}= (1\ \ldots\ N)$ in definition \ref{def:1.1}.
We require all other entries of $\mathbf{A}$, not supported on the $N$-cycle, to be zero:
\begin{equation}
f_{mm'}=0 \qquad \textrm{for} \ 1\leq m,m'\leq N, \ \textrm{unless} \  m' \in \{m, m-1\}\ \Mod N.
\label{eq:1.17}
\end{equation}
The blocks $\mathbf{A}', \mathbf{B}, \mathbf{C}, \mathbf{D}$ in the $M\times M$ Jacobian \eqref{eq:1.16} are assumed to remain bounded, uniformly in the small scaling parameter $\varepsilon >0$.
These assumptions make the $N$-cycle described by $\mathbf{A}$ dominant and \emph{fast} compared to the remaining dynamics, on the linear level.

The mathematical motivation for our emphasis on cycles, in addition to \cite{Hir1911}, comes from the Quirk-Ruppert-Maybee theorem; see the beautiful account in \cite{JKD77}.
That theorem addresses matrices $\mathbf{M}$ with prescribed sign structure of their entries. 
It characterizes spectral stability Re spec $ \mathbf{M} \leq 0$, for all $\mathbf{M}$ with the same sign structure, by three simultaneous requirements: nonpositive diagonal elements, nonpositive products over 2 cycles, and vanishing products $\beta$ over $N$-cycles, for $N \geq 3$.
Our results can be seen as an attempt to assert global Hopf bifurcation when the second conditions is violated, by an autocatalysis count $N_\mathrm{aut}>0$, and/or the third condition is violated, by fast $N$-cycles $\mathbf{A}$ with $\mathrm{sign} \ \beta = \pm1$.

The linearization
\begin{equation}
\dot{\xi} = \mathbf{A}\xi, \qquad \mathbf{A} = 
\left(
\begin{array}{cccc}
-a_1&&&\beta_1\\
\beta_2&-a_2&&\\
&\ddots&\ddots&\\
&&\beta_N&-a_N
\end{array}
\right),
\label{eq:1.18}
\end{equation}
on the fast $N$-cycle constitutes a linear \emph{cyclic monotone feedback system}.
See \cite{M-PS90} for a detailed spectral analysis and deep nonlinear consequences.
For simplicity we perform linear rescalings of each $\xi_m$\,, and a positive linear rescaling of time $t$, to normalize the off-diagonal $N$-cycle elements of $\mathbf{A}$ such that
\begin{equation}
\begin{aligned}
& \beta_2= \ldots =\beta_N = +1 \qquad \textrm{and} \\
& \beta = \prod_{m=1}^N \beta_m= \beta_1 = \pm1;\\
\end{aligned}
\label{eq:1.19}
\end{equation}
see \eqref{eq:1.13}.
Note here that the product $\beta$ of the $\beta_m$ is invariant under scalings of the $\xi_m$\,.
Under positive time rescaling the nonzero \emph{feedback sign} becomes $\mathrm{sign}\,\beta = \beta = \beta_1 = \pm1$.
We define the \emph{bifurcation parameter} $a>0$ by the normalization
\begin{equation}
a_m(a) = a \alpha_m, \qquad \prod_{m=1}^N \alpha_m = (-1)^{N_\mathrm{aut}}\,, \quad a^N = \mathrel{\bigg|} \prod_{m=n}^N a_m / \prod_{m=1}^N \beta_m \mathrel{\bigg|}\,,
\label{eq:1.20}
\end{equation}
along the original diagonal of the nondegenerate  $N$-cycle $\mathbf{A}$; see \eqref{eq:1.14}.
For the normalized diagonal elements $\alpha_m \neq 0$ we use the abbreviations $\langle \cdot \rangle$ and $\langle \cdot \rangle_h$ to denote their arithmetic and harmonic means, respectively.
We assume
\begin{equation}
\begin{aligned}
\langle \alpha \rangle &:= \frac{1}{N} \sum_{m=1}^N \alpha_m \neq 0, \quad &\textrm{i.e.}\ \sigma &:= \mathrm{sign} \langle \alpha \rangle \neq 0, \\
\langle 1/ \alpha \rangle &:= \frac{1}{N} \sum_{m=1}^N 1/\alpha_m \neq 0, &&\\
\langle \alpha \rangle_h &:= 1 /\langle 1 / \alpha \rangle \neq 0, \qquad \qquad &\textrm{i.e.}\ \sigma_h &:= \mathrm{sign} \langle \alpha \rangle_h \neq 0, \\
\end{aligned}
\label{eq:1.21}
\end{equation}
whenever $\sigma, \sigma_h$ appear.

In addition to the signs of the above arithmetic and harmonic means, our oscillation conditions will only involve the length $N\geqslant 3$ of the catalytic cycle, and the count $N_\mathrm{aut}$ of diagonal strongly autocatalytic entries $\alpha_m< 0$.
Specifically, we assume any one of the following four cases to hold.
\begin{enumerate}[label=(\roman*)]
\item For positive feedback $\beta = +1$ and $N \not\equiv 0 \  \Mod 4$:
\begin{equation}
\begin{aligned}
\phantom{|} N_\mathrm{aut} -2 \lfloor N/4\rfloor -1\phantom{|}\ &=\ \sigma_h\,, \quad \mathrm{or}\\
|N_\mathrm{aut} -2 \lfloor N/4\rfloor-1|\ &>\ 1. \\
\end{aligned}
\label{eq:1.22}
\end{equation}
\item For positive feedback $\beta = +1$ and $N\equiv 0 \ \Mod 4$:
\begin{equation}
\begin{aligned}
\phantom{|} N_\mathrm{aut} -2 N/4 +\sigma \phantom{|}\ &=\ \sigma_h\,, \quad \mathrm{or}\\
| N_\mathrm{aut} -2 N/4+ \sigma | \ &>\ 1. \\
\end{aligned}
\label{eq:1.23}
\end{equation}
\item For negative feedback $\beta = -1$ and $N \not\equiv 2 \ \Mod 4$:
\begin{equation}
\begin{aligned}
\phantom{|} N_\mathrm{aut} -2 \lfloor(N+2)/4\rfloor \phantom{|}\ &=\ \sigma_h\,, \quad \mathrm{or}\\
| N_\mathrm{aut}-2 \lfloor(N+2)/4\rfloor | \ &>\ 1; \\
\end{aligned}
\label{eq:1.24}
\end{equation}
\item For negative feedback $\beta = -1$ and $N \equiv 2 \ \Mod 4$:
\begin{equation}
\begin{aligned}
\phantom{|} N_\mathrm{aut} -2 N/4 +\sigma\phantom{|}\ &=\ \sigma_h\,, \quad \mathrm{or}\\
| N_\mathrm{aut} -2 N/4 + \sigma | \ &>\ 1. \\
\end{aligned}
\label{eq:1.25}
\end{equation}
\end{enumerate}

Here $\lfloor\cdot\rfloor$ denotes  the integer floor function.
We repeat that the signs $\sigma, \sigma_h = \pm1 $ of the arithmetic and harmonic means are assumed to be nonzero, respectively, whenever they appear. 
Specifically, this assumption will only arise under the following circumstances:
\begin{equation}
\begin{aligned}
& \sigma = \pm1 \quad & \textrm{for} & \ N \ \equiv 1-\beta \ \Mod 4,\\
& \sigma_h = \pm1 \quad &\textrm{for} & \ (-1)^{N_\mathrm{aut}}= \beta .\\
\end{aligned}
\label{eq:1.26}
\end{equation}

Concerning the open parameter interval $a \in (\underline{a}, \bar{a})$ where we will assert global Hopf bifurcation, we distinguish the following cases.
We fix
\begin{equation}
\begin{aligned}
& \underline{a} := 0 &\mathrm{for} &\ N \ \not \equiv 1-\beta\ \Mod 4,\\
& \underline{a} > 0 & \mathrm{for} & \ N  \ \equiv 1-\beta \ \Mod 4, \\
\end{aligned}
\label{eq:1.27}
\end{equation}
with arbitrarily small $\underline{a} $ in the second case.
Similarly, we fix
\begin{equation}
\begin{aligned}
& \bar{a} := + \infty &\mathrm{for} & \ (-1)^{N_{\mathrm{aut}}} = -\beta,\\
& \bar{a} < 1 & \mathrm{for} & \ (-1)^{N_{\mathrm{aut}}} = \beta, \\
\end{aligned}
\label{eq:1.28}
\end{equation}
with arbitrarily small $1-\bar{a}$ in the second case.

\begin{thm} \label{thm:1.2}
Consider a network \eqref{eq:1.10} with a fast nondegenerate $N$-cycle $\mathbf{A}$ on $\mathbf{m}= (1 \ \ldots\  N),\  N\geq 3$, of the Jacobian \eqref{eq:1.16} -- \eqref{eq:1.18} at the $(\varepsilon, a)$-independent steady state $x^*$ in \eqref{eq:1.15}.
Assume hyperbolicity of the lower right block $\mathbf{D}$ in the Jacobian \eqref{eq:1.16}, i.e. $0\, \not\in \, \mathrm{Re}\,  \mathrm{spec}\,  \mathbf{D}$.
Let the normalization \eqref{eq:1.19} and parameter assumptions \eqref{eq:1.20}, \eqref{eq:1.26} -- \eqref{eq:1.28}, on $a>0$ hold.

Then each of the cases \eqref{eq:1.22} -- \eqref{eq:1.25} leads to the following conclusion.
There exists $\varepsilon_0 >0$ depending on $\underline{a}, \bar{a}$ such that for any fixed $0 < \varepsilon < \varepsilon_0$ the network \eqref{eq:1.10} exhibits global Hopf bifurcation of nonstationary periodic solutions from the steady state $x^*$, for parameters $a \in (\underline{a}, \bar{a})$.
\end{thm}

For our precise notion of global Hopf bifurcation we refer to definition \ref{def:2.2} and corollary \ref{cor:2.4} below.

Let us rewrite the four options \eqref{eq:1.22} -- \eqref{eq:1.25} in more concise form.
Specifically, we define the abbreviating signs
\begin{eqnarray}
\iota = \iota(N) &:=&
\begin{cases}
     +1 & \text{for}\ N\equiv 1 \\
     -1 & \text{for}\ N\equiv -1\\
     \phantom{+}0  & \text{for}\ N\equiv 0,2
\end{cases} 
\quad\ \ \Mod 4;
\label{eq:1.33}\\
\kappa = \kappa(N,\beta) &:=&
\begin{cases}
      \phantom{+}0 & \text{for}\ N\not\equiv 1-\beta \\
      \phantom{+}1 & \text{for}\ N\equiv 1-\beta
\end{cases} 
\quad \Mod 4.
\label{eq:1.34}
\end{eqnarray}
Note the isomorphism $\iota$ from the group $\mathbb{Z}_4^*$ of multiplicative units in $\mathbb{Z}_4$ to the multiplicative group $\mathbb{Z}_2=\{\pm 1\}$.
The four options \eqref{eq:1.22} -- \eqref{eq:1.25} can then be condensed into the single assumption
\begin{equation}
\begin{aligned}
\phantom{|} 2 N_\mathrm{aut} - N - \iota\beta + 2 \kappa\sigma \phantom{|} \ &=\ 2 \sigma_h\,, \quad \mathrm{or}\\
|2 N_\mathrm{aut} - N - \iota\beta + 2 \kappa\sigma | \ &> \ \ 2\,. \\
\end{aligned}
\label{eq:1.35}
\end{equation}

For a first consistency check of our result, consider time reversal $t\mapsto -t$ in our original ODE setting \eqref{eq:2.1}.
Stationary solutions, periodic orbits (as sets), and global Hopf bifurcation remain unaffected by time reversal. 
But how about the options \eqref{eq:1.22} -- \eqref{eq:1.25} of theorem \ref{thm:1.2}, alias assumption \eqref{eq:1.35}?

Under time reversal, all $\alpha_m\,,\, \beta_m$ reverse sign, and so do the signs $\sigma, \sigma_h$ of the arithmetic and harmonic means $\langle\alpha \rangle , \langle\alpha \rangle_h$\,.
Sign reversal of all $\alpha_m$ therefore replaces the count $N_\mathrm{aut}$ of strongly autocatalytic $\alpha_m<0$ in the $N$-cycle by $N-N_\mathrm{aut}$\,. 
This reverses the sign of the term $2 N_\mathrm{aut} - N$ in assumption \eqref{eq:1.35}.
The normalized feedback coefficient $\beta$ gets multiplied by $(-1)^N$.
By abbreviation \eqref{eq:1.33}, however, which eliminates even $N$, this also just reverses the sign of $\iota\beta$ in assumption \eqref{eq:1.35}.
The coefficient $\kappa = \kappa(N,\beta)$, in contrast, vanishes for odd $N$ and thus remains unaffected by time reversal.
In summary, these elementary considerations show how assumption \eqref{eq:1.35} remains invariant under time reversal -- just as the conclusion of global Hopf bifurcation in theorem \ref{thm:1.2} does.


\section{Global Hopf bifurcation}\label{sec2}
We introduce the main tools in our analysis of autonomous time periodic oscillations.
Skipping proofs, we adapt results on global Hopf bifurcation going back to \cite{Fie85}, and based on earlier results by Jim Yorke and others \cite{AY78, M-PY82, AllY84}.
Specifically, we introduce virtual periods, and the center index $\zhong$ (pronounced ``zhong").
Our main abstract results are summarized in theorem \ref{thm:2.3} and corollary \ref{cor:2.4} below. 
See also \cite{Fie88}, section 3, for a more detailed survey.

In this section we consider general vector fields
\begin{equation}
\dot{x} = f(a, x)
\label{eq:2.1}
\end{equation}
on $x \in \mathbb{R}^M$, with scalar parameter $a \in \mathbb{R}$, and continuous $f, f_x$.
As before, we call $x^*$ \emph{stationary} at $a$, or \emph{steady state}, if $f(a, x^*)= 0$.
We call a solution $x(t)$ \emph{periodic} with a \emph{period} $T>0$ if $x(t)$ is nonstationary and 
\begin{equation}
x(t+T)=x(t)
\label{eq:2.2}
\end{equation}
holds for all real $t$.
The set of all periods $T$ then takes the form $T=kp$, with $k \in \mathbb{N}$, where $p>0$ is called the \emph{minimal period} of $x(t)$.

Global bifurcation results are often based on topological tools which, in turn, are established via approximations and limit arguments.
For periodic solutions, a peculiar difficulty arises from the notion of minimal period $p$, versus periods $T=kp$: the limits of minimal periods may fail to be minimal periods of the limit.
At Hopf bifurcation, for example, we observe sequences of periodic orbits shrinking to a stationary solution, although the minimal periods approach a positive limit.
At period doubling, we observe sequences of periodic orbits where the  minimal period drops by a factor two, in the limit.
More generally, the same difficulty arises in systems with group equivariance: isotropies may increase (but not decrease) under limits.
This more encompassing equivariant viewpoint, which we do not pursue any further here, has been addressed in \cite{Fie88}, section 4.

The notion of virtual periods remedies the limit deficiency of minimal periods, by including the associated linear flow.
We call $q>0$ a \emph{virtual period} of $x(t)$ at $a$, if $q$ is the minimal period of some pair $(x(t),y(t))$, where $y(t)$ satisfies the (nonautonomous) linearized equation
\begin{equation}
\dot{y}(t)= f_x(a, x(t))y(t),
\label{eq:2.3}
\end{equation}
for all real $t$.
We also use this terminology if $x(t)\equiv x^*$ happens to be stationary.
In summary, virtual periods are the minimal periods of the induced flow on the tangent bundle.

Stationary $x^*$ with virtual periods $q$ are called \emph{Hopf points}: indeed they possess nonzero purely imaginary eigenvalues.
Standard Hopf bifurcation, from a transverse crossing of a pair of simple and nonresonant eigenvalues $\pm \mathrm{i}$ of $f_x(a, x^*)$, is indicative of a virtual period $q=2\pi$ at $x^*$.
Any subset of rationally related eigenvalues $\pm \mathrm{i} \omega$ generates a virtual period $q$ given by the least common multiple of the $2\pi/\omega$.
Standard period doubling at minimal period $p$ features two virtual periods $q$: $p$ itself, and $2p$.
Note how the set of virtual periods $q$ of any fixed stationary or periodic solution $x(t)$ is always bounded, whereas the set $T=kp$ of all periods is always unbounded.

It turns out that virtual periods, unlike minimal periods, are closed under limits.
\begin{prop}\label{prop:2.1}
Let $q_n$ be a virtual period of $x_n$ at parameter $a_n$.
Assume bounded convergence: 
\begin{equation}
(a_n, x_n, q_n)\rightarrow (a_\infty, x_\infty, q_\infty).
\label{eq:2.4}
\end{equation}
Then $q_\infty>0$, and $q_\infty$ is a virtual period of $x_\infty$ at parameter $a_\infty$.
\end{prop}

Henceforth we require all Hopf points $(a_*, x^*)$ of \eqref{prop:2.1} to be nondegenerate, i.e.
\begin{equation}
\det\, f_x (a_*, x^*)\neq 0.
\label{eq:2.5}
\end{equation}
This allows us to continue the steady state $x^*= x^* (a)$, locally, by the implicit function theorem.
Let
\begin{equation}
\mu (a) :=  \# \{\mathrm{ Re \,spec} \, f_x(a, x^*(a))>0\}
\label{eq:2.6}
\end{equation}
count the strictly unstable eigenvalues at $(a, x^*(a))$, with algebraic multiplicity.
We now require Hopf points to be isolated, in $\mathbb{R}\times \mathbb{R}^M$.
Then $\mu(a)$ is the unstable dimension, or Morse index, of the hyperbolic steady state $x^*(a)$, for nearby $a \neq a_*\ $.
This allows us to define the \emph{crossing number}
\begin{equation}
\chi(a_*):= \tfrac{1}{2} \lim_{\delta\searrow 0} \ (\mu (a_* + \delta)- \mu (a_* - \delta))= \tfrac{1}{2}(\mu (a_*^+)- \mu (a_*^-))
\label{eq:2.7}
\end{equation}
of the Hopf point $x^*$ at $a=a_*$\,.
This is the net number of eigenvalue pairs crossing the imaginary axis from left to right, as $a$ increases through $a_*$\,.
Finally, following \cite{M-PY82}, we define the \emph{center index} of the Hopf point $(a_*, x^*)$ as
\begin{equation}
\zhong (a_*, x^*):= (-1)^{\mu(a_*)}\cdot \chi (a_*)
\label{eq:2.8}
\end{equation}
Fix any open subset $\mathcal{U}$ of $\mathbb{R} \times \mathbb{R}^M$, such that $\mathcal{U}$ contains the whole nonstationary periodic orbit, with any point on it.
We clarify our notion of global versus local continua of periodic solutions and Hopf points in $\mathcal{U}$ as follows.
Denote
\begin{equation}
\begin{aligned}
& \mathcal{Q}:= \{(a, x, q) \,|\, q>0 \mathrm{ \ is \ a\ virtual\ period\ of \ } (a,x) \in \mathcal{U} \} \,,\\
& \mathcal{P} := \{(a,x) \,|\, (a,x,q)\in \mathcal{Q} \}\,.\\
\end{aligned}
\label{eq:2.9}
\end{equation}
In other words, $\mathcal{P} = \check{q} \tilde{Q}  $, where the projection $\check{q}$ omits the $q$-component of $\mathcal{Q}$.
\begin{defi}\label{def:2.2}
A connected component $\mathcal{C}$ of $\mathcal{P}$, i.e. of the periodic solutions and Hopf points in $\mathcal{U}$, is called \emph{local} in $\mathcal{U}$, if the closure of $\mathcal{C}$ is compactly contained in $\mathcal{U}$ and the virtual periods in $\mathcal{C}$ are bounded above. 
In other words, the lift $\check{q}^{-1} \mathcal{C}$ is compactly contained in $\mathcal{Q}$.
Connected components $\mathcal{C}$ which are not local are called \emph{global}.
\end{defi}
Note that proposition \ref{prop:2.1} asserts compactness of local components $\mathcal{C}$.
\begin{thm} \label{thm:2.3}
Consider the flow \eqref{eq:2.1} and assume all Hopf points in $\mathcal{U}$ are nondegenerate, as in \eqref{eq:2.5}, and isolated.
Let $\mathcal{C}$ be a local connected component of the periodic solutions and Hopf points $\mathcal{P}$ in $\mathcal{U}$.
Then
\begin{equation}
\sum_\mathcal{C} \, \zhong=0, \\
\label{eq:2.10}
\end{equation}
where the sum ranges over the finitely many Hopf points in $\mathcal{C}$, if any.
\end{thm}
\begin{cor}\label{cor:2.4}
In the setting of theorem \ref{thm:2.3}, assume $\mathcal{P}$ contains only finitely many Hopf points and
\begin{equation}
\sum_\mathcal{P} \, \zhong \neq 0. \\
\label{eq:2.11}
\end{equation}
Then $\mathcal{P}$ possesses at least one global connected component $\mathcal{C}$ which also contains a Hopf point.
We call this case \emph{global Hopf bifurcation in} $\mathcal{U}$.
\end{cor}
To derive the corollary from the theorem, let $\mathcal{C_\ell}$ enumerate the finitely many disjoint connected components of $\mathcal{P}$, which contain Hopf points.
Suppose, indirectly, that each $\mathcal{C_\ell}$ is local.
Then \eqref{eq:2.10} implies
\begin{equation}
\sum_{\mathcal{P}}\, \zhong = \sum_\ell \sum_{\mathcal{C_\ell}}\, \zhong = 0,
\label{eq:2.12}
\end{equation}
contradicting \eqref{eq:2.11}.
Hence at least one $\mathcal{C}_{\ell}$ is global, by theorem \ref{thm:2.3}.

The proof of theorem \ref{thm:2.3} is based on generic approximation.
The cancellation \eqref{eq:2.10} of center indices on compact connected components $\mathcal{C}$ follows from the same property in the generic situation, by approximation. 
See \cite{M-PY82} for the generic case.
This requires a parametrized version of the Kupka-Smale theorem, via Thom-transversality, and a detailed degree argument which carefully distinguishes periodic orbits with orientable and nonorientable unstable manifolds.
The resulting global Hopf components of orientable periodic orbits in $\mathcal{P}$ were called \emph{snakes}, in \cite{M-PY82}.
The only discontinuities of minimal periods, in the generic case of snakes, occur at period doubling bifurcations.
By generic approximation, this reveals jumps by factors 2 as the only possible discontinuities of virtual periods, in the general case of nongeneric snake limits.

Competing topological results, based on the $J$-homomorphism or $S^1$-equivariant degree theory, studied continua of triples $(a, x, T)$ with (not necessarily minimal) periods $T$ of $(a, x)$ as in \eqref{def:2.2}; see \cite{AY78} for the original, and \cite{IV03} for more recent developments with many references.
The intriguing ``jug-handle" example by \cite{AllY84} exhibits a continuum with bounded $(a,x)$ and unbounded $T$, whereas virtual periods remain bounded.
The jug-handle consists of a compact loop of periodic orbits $(a,x)$, where the two branches generated at a saddle-node bifurcation re-unite, at a period-doubling.
Any such loop generates an unbounded continuum of triples $(a,x,T)$ where $T$ traverses all multiples $2^kp$ of the minimal periods $p$, for $k \in \mathbb{N}_0$.
The jug handle, which is evidently quite bounded in parameter and phase space, therefore counts as a global, unbounded continuum in the sense of \cite{AY78, IV03}.
The virtual periods on the jug handle, however, remain bounded: they are given by $p$ and, at the period doubling only, by $\{p,2p\}$.
The same jug handle therefore does not qualify as global, in the sense of \cite{M-PY82, Fie85} which we adopt in the present paper.

Of course our notion of ``globality" depends on the choice of the underlying open domain $\mathcal{U}\subseteq \mathbb{R} \times \mathbb{R}^M$ where we study our continua.
Indeed we can only assert a \emph{global trichotomy} for any global component $\mathcal{C}\subseteq \mathcal{U}$:\newline
\begin{minipage}[t]{.1\textwidth}
\raggedright
\vspace{0.85cm}
\begin{equation}\label{eq:2.15}
\;
\end{equation}
\end{minipage}
\begin{minipage}[t]{.9\textwidth}
\raggedleft
\begin{enumerate}[label=(\roman*)]
\item either $\mathcal{C}$ is unbounded, or
\item $\mathcal{C}$ is bounded, but $\partial \mathcal{C}\cap \partial \mathcal{U} \neq \emptyset $, or else 
\item $\mathrm{clos} \ \mathcal{C}$ is compactly contained in $\mathcal{U}$ with unbounded virtual periods.
\end{enumerate}
\end{minipage}

\medskip
Option (iii) of the global trichotomy \eqref{eq:2.15} is particularly interesting.
For example, consider a convergent sequence $(a_n,x_n)\rightarrow (a_\infty,x_\infty)$ of Hopf points with purely imaginary eigenvalues $\pm i\omega_n$, such that $\omega_n\searrow 0$.
The steady state $x_\infty$ then features an eigenvalue $\omega_\infty=0$ with algebraic multiplicity at least two; in the simplest interesting case this is a Bogdanov-Takens point.
The virtual periods $q_n:= 2 \pi / \omega_n$ converge to $+ \infty$, of course.

More generally, suppose $(a_n, x_n)$ are nonstationary periodic with minimal periods $p_n\rightarrow +\infty$.
Suppose $(a_n,x_n)\rightarrow (a_\infty , x_\infty)$ becomes stationary, in the limit, but some part of the periodic orbit $x_n(t)$ of $x_n$ does not converge to $x_\infty$.
In the simplest interesting case this may happen by convergence of $x_n(.)$ to a homoclinic orbit attached to the steady state $x_\infty \,$.
This example is closely related to the Takens-Bogdanov case, which generates small amplitude homoclinic orbits.
For global consequences in vector fields with two real parameters see \cite{Fie86,Fie96}.

Suppose next that the orbits $x_n(.)$ remain bounded, and stay away from any steady states.
Remarkably $p_n\rightarrow \infty$ can still occur, along a continuum of periodic orbits and without any bifurcations affecting the minimal periods $p_n$.
In 1995, such \emph{blue sky catastrophes} have first been constructed by Turaev and Shilnikov, in a structurally stable way involving a single parameter.
See the survey \cite{SST14}.

In section 4, we will apply corollary \ref{cor:2.4} to the situation of theorem \ref{thm:1.2}.
In particular we note how the crossing numbers in \eqref{eq:2.11} simply add up to a net crossing number, along a steady state $x^*$ which does not actually depend on $a$, as long as $x^*$ remains nondegenerate.
To account for the slow-fast dichotomy \eqref{eq:1.16} of the linearization $f_x(\varepsilon, a, x^*)$ we will also narrow attention from $a \in (-\infty, +\infty)$ to $0<a \in (\underline{a}, \bar{a})$.


\section{Linear feedback cycles}\label{sec3}
In this section we collect some spectral properties of the normalized nondegenerate $N$-cycle 
\begin{equation}
\mathbf{A} = \mathbf{A}(a) = 
\left(
\begin{array}{cccc}
-a \alpha_1&&&\beta\\
1 &-a\alpha_2&&\\
&\ddots&\ddots&\\
&&1&-a\alpha_N
\end{array}
\right),
\label{eq:3.1}
\end{equation}
with $a>0,\ \prod \alpha_m=(-1)^{N_{\mathrm{aut}}}\,$, and $\beta= \pm1$.
See \eqref{eq:1.16} -- \eqref{eq:1.21} and theorem \ref{thm:1.2}.

Proposition \ref{prop:3.1} recalls the general pairwise ordering of the eigenvalues $\lambda_k$ of $\mathbf{A}$ by their real parts, due to \cite{M-PS90}.
Proposition \ref{prop:3.2} addresses crossings of eigenvalues through the imaginary axis, at $a=0$ and $a=1$.
Proposition \ref{prop:3.3} collects the limits, at $a=0^+, 1^\pm$, and $\infty$, of the strict unstable dimensions $\mu (a)$ introduced in \eqref{eq:1.21}:
\begin{equation}
\mu(a_0^\pm):= \lim \mu (a), \quad \mathrm{for} \ 0< \pm (a-a_0)\searrow 0.
\label{eq:3.2}
\end{equation}
We conclude in proposition \ref{prop:3.4}, by showing how  the presence of a zero eigenvalue, at $a=1$, prevents all further purely imaginary eigenvalues to occur for any $a \geq 1$.

The \emph{zero number} $z(\xi)$, an integer-valued Lyapunov function for $\dot{\xi}= \mathbf{A}\xi$, is the crucial tool in the deep spectral (and nonlinear) analysis of \cite{M-PS90}.
In our normalization \eqref{eq:3.1}, consider positive feedback $\beta=+1$ first and let $0 \neq \xi \in \mathbb{R}^N$.
Then $z(\xi)\geq 0$ denotes the (even) number of strict sign changes in the ordered cyclic sequence of $\xi$-components $\xi_m$, with $m \,\Mod N$.
For negative feedback $\beta = -1$, however, we modify that count between $\xi_N$ and $\xi_1$, only, to account for a strict sign change between $\beta \xi_N= -\xi_N$ and $\xi_1$, instead.
In particular $z(\xi)\geq 1$ becomes odd.
In summary we obtain
\begin{equation}
(-1)^{z(\xi)} = \beta,
\label{eq:3.3}
\end{equation}
for both feedback cases, $\beta = \pm1$.

\begin{prop}\label{prop:3.1}
Assume negative feedback, $\beta= -1$.
Then the eigenvalues $\lambda_k$ of $\mathbf{A}$ can be ordered in pairs, repeated with algebraic multiplicity, such that
\begin{equation}
\mathrm{Re}\lambda_0\geq \mathrm{Re}\lambda_1> \mathrm{Re}\lambda_2 \geq \mathrm{Re}\lambda_3> \ldots\,,
\label{eq:3.4}
\end{equation}
for indices ranging from $0$ to $N-1$.
Simple real eigenvalues are labeled in strictly decreasing order.
In particular, all eigenvalues are algebraically simple, except for some possibly double real eigenvalues.

The associated real eigenvectors $\xi_k$ of $\lambda_k$ can be chosen to satisfy
\begin{equation}
z(\xi_{2k})=z(\xi_{2k+1})= 2k+1,
\label{eq:3.5}
\end{equation}

Here $\xi_{2k}$ and  $\xi_{2k+1}$ refer to the real and imaginary parts of the complex eigenvectors, in case $\lambda_{2k+1}= \bar{\lambda}_{2k}$ are nonreal conjugate complex.

For positive feedback $\beta = +1$, the analogous ordering reads 
\begin{equation}
\lambda_0 > \mathrm{Re} \lambda_1\geq \mathrm{Re} \lambda_2 > \mathrm{Re} \lambda_3 \geq \mathrm{Re} \lambda_4 > \ldots \  .
\label{eq:3.6}
\end{equation}
with real eigenvectors $\xi_k$ of $\lambda_k$ satisfying
\begin{equation}
z(\xi_{2k-1})=z(\xi_{2k})= 2k,
\label{eq:3.7}
\end{equation}
for resulting indices in $\{0,\ldots,N-1\}$.
\end{prop}
\begin{proof}[\textbf{Proof.}]
See \cite{M-PS90}.
\end{proof}
To get slightly more specific we write the characteristic equation for the characteristic polynomial $\mathbf{p}$ of $\mathbf{A}$ from \eqref{prop:3.1} as 
\begin{equation}
\begin{aligned}
0 = \mathbf{p} &= \det(\lambda-\mathbf{A}(a))= \prod_{m=1}^N(\lambda+ a\alpha_m)-\beta = \\
&= \lambda^N + \langle \alpha \rangle N a\lambda^{N-1}+\ldots+
(-1)^{N_{\mathrm{aut}}} \langle 1/ \alpha \rangle N a^{N-1} \lambda + (-1)^{N_{\mathrm{aut}}} a^N- \beta \,.
\end{aligned}
\label{eq:3.8}
\end{equation}
Here we have used the normalizations \eqref{eq:1.20}, and the notation \eqref{eq:1.21} for arithmetic means.
The case $a=0$ in \eqref{eq:3.8}, with the $N$-th roots of unity $\lambda_k^N = \beta = \pm1,\ k=0, \ldots, N-1$, as simple eigenvalues, provides an instructive example for the two feedback cases of proposition \ref{prop:3.1}.
\begin{prop}\label{prop:3.2}
Consider the normalized nondegenerate $N$-cycle $\mathbf{A}(a)$ of \eqref{eq:3.1}, for $a\geq 0$.
Then the following holds true.
\begin{enumerate}[label=(\roman*)]
\item An eigenvalue $\lambda_k=0$ occurs if, and only if,
\begin{equation}
(-1)^{N_{\mathrm{aut}}}= \beta \quad \mathrm{and} \quad a=1.
\label{eq:3.9}
\end{equation}
The eigenvector $\xi_k$ of $\lambda_k=0$ satisfies
\begin{equation}
z(\xi_k)= N_{\mathrm{aut}}\,.
\label{eq:3.10}
\end{equation}
The eigenvalue $\lambda_k=0$ is simple if, and only if, 
\begin{equation}
\langle 1/\alpha \rangle \neq 0.
\label{eq:3.11}
\end{equation}
In that case, $\lambda_k(a)$ crosses the imaginary axis transversely, at $a=1$, with nonzero derivative
\begin{equation}
\lambda'_k(1)= -1/ \langle 1/\alpha \rangle = - \langle \alpha \rangle _h
\label{eq:3.12}
\end{equation}
given by the harmonic mean; see \eqref{eq:1.21}.
\item At $a=0$, the eigenvalues $\lambda_k(a)$ are given by the $N$ simple roots of unity
\begin{equation}
\lambda_k^N= \beta = \pm1,
\label{eq:3.13}
\end{equation}
with $k=0,\ldots,N-1$.
Their derivatives with respect to $a$, at $a=0$, are all equal, given by the arithmetic mean
\begin{equation}
\lambda'_k(0) = - \langle \alpha \rangle .
\label{eq:3.14}
\end{equation}
In particular, the purely imaginary eigenvalues $\lambda_k= \pm i$ which occur for $N \equiv 0,2 \,\Mod 4$ and $\beta = +1,-1$, respectively, cross the imaginary axis transversely, for arithmetic means $\langle \alpha \rangle \neq 0$.
\end{enumerate}
\end{prop}
\begin{proof}[\textbf{Proof.}]
We use expansion \eqref{eq:3.8} of the characteristic polynomial.

To prove claim \eqref{eq:3.9} of (i), we just insert $\lambda = 0$ in \eqref{eq:3.8} and recall $a \geq 0$.
Algebraic simplicity claim \eqref{eq:3.11} is equally obvious, for $\langle 1/ \alpha \rangle \neq 0$. 
Implicit differentiation of $\mathbf{p}(a,\lambda(a))=0$ at $a=1,\ \lambda(1)=0$ yields
\begin{equation}
0= \mathbf{p}_a + \mathbf{p}_\lambda \lambda' = (-1)^{N_{\mathrm{aut}}} Na^{N-1}+ (-1)^{N_{\mathrm{aut}}} \langle 1/\alpha \rangle N a ^{N-1}  \lambda',
\label{eq:3.15}
\end{equation}
which proves \eqref{eq:3.12}.

To prove \eqref{eq:3.10} note that any eigenvector $0 \neq \xi \in \ker \mathbf{A}$ of $\lambda= 0$ at $a=1$ satisfies
\begin{equation}
\alpha_m \xi_m = \beta_m \xi_{m-1}
\label{eq:3.16}
\end{equation}
for $m \,\Mod N, \ \beta_1 = \beta$, and $\beta_2 =\ldots=\beta_N=1$; see \eqref{prop:3.1}.
Consider the case $\beta = +1, \ \alpha_1>0$, first.
Then those $2 \leq k \leq N$ with strongly autocatalytic $\alpha_k<0$ indeed provide precisely $N_\mathrm{aut}$ sign changes in the cyclic sequence of $\xi_m$, for $m\, \Mod N$.
The remaining cases of \eqref{eq:3.10} are proved analogously.
This proves claim (i).

To prove claim (ii), we insert $a=0$ in $\eqref{eq:3.8}$ to obtain the algebraically simple $N$-th roots of unity $\lambda_k^N= \beta, \ k=0, \ldots,N-1$.
Implicit differentiation of $\mathbf{p}(a, \lambda(a))=0$ at $a=0, \ \lambda^N= \beta \neq 0$, indeed yields
\begin{equation}
0= \mathbf{p}_a+ \mathbf{p}_\lambda \lambda' = \langle \alpha \rangle N  \lambda^{N-1}+N \lambda^{N-1} \lambda'.
\label{eq:3.17}
\end{equation}
This completes the proof of the proposition.
\end{proof}

\begin{prop}\label{prop:3.3}
At $a=\infty$ we obtain the following limiting strict unstable dimension:%
\begin{equation}
\mu(\infty):= \lim_{a\rightarrow \infty} \  \mu(a)= N_{\mathrm{aut}}\,.
\label{eq:3.18}
\end{equation}
Assume any $\mathbf{A}(a),\ a\geq 0$, possesses an eigenvalue $\lambda_k=0$, in the ordering of proposition \ref{prop:3.1}.
Then $(-1)^{N_{\mathrm{aut}}}= \beta$ and $a=1$, by \eqref{eq:3.9}.

If we also assume $\langle 1/ \alpha \rangle \neq 0$, as in \eqref{eq:1.21}, so that the harmonic mean with sign $\sigma_h= \mathrm{sign} \langle \alpha \rangle_h = \pm1$ exists, then  the limiting strict unstable dimensions $\mu(1^\pm)$ in \eqref{eq:3.2} are
\begin{equation}
\mu(1^\pm)= k+ \tfrac{1}{2} (1 \mp \sigma_h)\,,
\label{eq:3.19}
\end{equation}
and the even/odd parity of $k$ is determined by
\begin{equation}
(-1)^k = \beta \sigma_h \,.
\label{eq:3.20}
\end{equation}
In the limit $a \searrow 0$ and for $N \not\equiv 1- \beta \ \Mod 4,\  \beta = \pm1$, we obtain
\begin{equation}
\mu(0^+)= 2\lfloor(N-1+ \beta)/4\rfloor+1+(1-\beta)/2 \,.
\label{eq:3.21}
\end{equation}
For $N \equiv 1- \beta \ \Mod4$ we assume $\sigma = \mathrm{sign} \langle \alpha \rangle \neq 0$, as in \eqref{eq:1.21}, and obtain
\begin{equation}
\mu(0^+)= 2(N-1+ \beta)/4- \sigma +(1-\beta)/2 \,.
\label{eq:3.22}
\end{equation}
\end{prop}
\begin{proof}[\textbf{Proof.}]
To prove $\mu(\infty)= N_{\textrm{aut}}$ we invoke the characteristic equation \eqref{eq:3.8}, once again.
Trivially, $-\alpha_m \neq 0$ with $m=1,\ldots,N$ enumerate the limits of $\lambda_k(a)/a$ with $k=0,\ldots,N-1$, for $a\rightarrow+ \infty$.
This proves claim \eqref{eq:3.18}.

To prove claim \eqref{eq:3.19} we consider the simple eigenvalue $\lambda_k=0$ at $a=1$ with eigenvector $\xi_k$ and $z(\xi_k)= N_{\mathrm{aut}}$, from proposition \ref{prop:3.2}.
In particular, the ordering of $\mathrm{Re}\,\lambda_k$ in proposition \ref{prop:3.1} implies $\mu(a)=k$ for the \emph{strict} unstable dimension $\mu$ at $a=1$.
Our assumption $\langle 1/ \alpha \rangle \neq 0$ in \eqref{eq:1.21} also implies transverse crossing \eqref{eq:3.12} of $\lambda_k(a)$, at $a=1$, so that%
\begin{equation}
\mathrm{sign} \ \lambda_k(a)=\sigma_h \cdot \mathrm{sign}(1-a),
\label{eq:3.23}
\end{equation}
for small $|1-a|>0$.
Because proposition \ref{prop:3.1} excludes any other purely imaginary eigenvalues at $a=1$, besides the simple eigenvalue $\lambda_k(a)=0$, this proves claim \eqref{eq:3.19}.

To prove \eqref{eq:3.20} we observe that the absence of zero eigenvalues for $1<a<\infty$ implies that $\mu(1^+)$ and $\mu(\infty)= N_{\mathrm{aut}}$ share the same parity.
Therefore \eqref{eq:3.19} implies \eqref{eq:3.20} via%
\begin{equation}
\beta=(-1)^{N_{\mathrm{aut}}}= (-1)^{\mu(1^+)}= (-1)^k \cdot (-1)^{(1-\sigma_h)/2}=(-1)^k \sigma_h .
\label{eq:3.24}
\end{equation}

It remains to consider $\mu(0^+)$ with eigenvalues $\lambda$ at $a=0$ given by the simple roots of unity $\lambda^N= \beta = \pm 1$.
For the strict unstable dimension $\mu(0)$, which ignores purely imaginary eigenvalues, elementary counting shows $\mu(0)=2\lfloor(N-2+ \beta)/4\rfloor+1+ (1- \beta)/2$.
For $N \not\equiv 1-\beta\ \Mod4$, purely imaginary roots $\lambda_k=\pm i$ do not occur.
Therefore $\lfloor(N-2+ \beta)/4\rfloor = \lfloor(N-1+ \beta)/4\rfloor$ proves \eqref{eq:3.21}.
For $N \equiv 1- \beta\ \Mod 4$, the purely imaginary pair (as all other roots) satisfies $\lambda'_k = \mathrm{Re}\, \lambda'_k = - \langle \alpha \rangle$; see \eqref{eq:3.14}.
Therefore
\begin{equation}
\mathrm{sign}\  \mathrm{Re}\, \lambda'_k(a)= -\sigma,
\label{eq:3.25}
\end{equation}
for small $a>0$, and hence $\mu(0^+)=\mu(0)+1-\sigma$.
Insertion of our elementary count for $\mu(0)$ proves \eqref{eq:3.22}, and the proposition.
\end{proof}
\begin{prop}\label{prop:3.4}
As in \eqref{eq:3.19}, \eqref{eq:3.20} suppose any $\mathbf{A}(a), \ a\geq 0$ possesses an eigenvalue $\lambda_k=0$, i.e. $(-1)^{N_{\mathrm{aut}}} = \beta$ and $a=1$.
Assume $\langle 1/ \alpha \rangle \neq 0$.

Then $\mathbf{A}(a)$ does not possess any other zero or purely imaginary eigenvalues, for any $1 \leq a< \infty$, except that simple zero eigenvalue $\lambda_k=0$ at $a=1$.
\end{prop}
\begin{proof}[\textbf{Proof.}]
Propositions \ref{prop:3.1} and \ref{prop:3.2}(i) establish the claim at $a=1$.
We have to show that purely imaginary nonzero eigenvalues cannot occur, for any $a>1$.
We only consider the case of positive feedback, $\beta = +1$; the case $\beta = -1$ is analogous.

Since $(-1)^{N_{\mathrm{aut}}}= \beta$, positive feedback $\beta = +1$ implies $N_{\mathrm{aut}}$ is even.
Therefore \eqref{eq:3.10} implies
\begin{equation}
z(\xi_k)= N_{\mathrm{aut}}= 2k',
\label{eq:3.26}
\end{equation}
for the eigenvector $\xi_k$ of $\lambda _k=0$ at $a=1$.
The ordering \eqref{eq:3.6}, \eqref{eq:3.7} of real eigenvalues implies 
\begin{equation}
k \in \{2k'-1,\ 2k'\}
\label{eq:3.27}
\end{equation}
for the two real simple eigenvalues $\lambda_{2k'-1}> \lambda_{2k'} $, one of them being zero, at $a=1$.
We claim $\lambda_{2k'-1}$ and $\lambda_{2k'}$ straddle zero, for all $a>1$: both eigenvalues remain simple, real, and satisfy
\begin{equation}
\lambda_{2k'-1}> 0>\lambda_{2k'} \ .
\label{eq:3.28}
\end{equation}
Then the straddling eigenvalues  $\lambda_{2k'-1}\,, \lambda_{2k'}\,,$ in view of the ordering \eqref{eq:3.6}, \eqref{eq:3.7}, prevent any other real or complex eigenvalues from crossing the imaginary axis, at any $a>1$, and the proposition will be proved.

We prove our remaining claim \eqref{eq:3.28} for $\sigma_h=+1$; the case $\sigma_h=-1$ is analogous.
Parity property \eqref{eq:3.20}, $(-1)^k= \beta \sigma_h=+1$, asserts $k$ is even.
Hence $k=2k'$ in \eqref{eq:3.27}, and transverse crossing \eqref{eq:3.12} implies
\begin{equation}
\lambda_{2k'-1}> 0>\lambda_k=\lambda_{2k'} \ ,
\label{eq:3.29}
\end{equation}
for small $a-1>0$.
Absence of zero eigenvalues, for $a>1$, together with the strict ordering and pairing of proposition \ref{prop:3.1}, \eqref{eq:3.6}, preserves simplicity of the real eigenvalues and perpetuates \eqref{eq:3.29} to all real $a>1$.
This proves the proposition.
\end{proof}


\section{Main result: proof}\label{sec4}
In this section we return to the original setting
\begin{equation}
\begin{aligned}
\dot{x}& =f(\varepsilon,a,x)\,,\\
0 & = f(\varepsilon,a,x^*)\,,\\
\mathcal{A}(\varepsilon, a) & = f_x(\varepsilon,a,x^*)= 
\left(
\begin{array}{cc}
\mathbf{A}(a)+ \varepsilon \mathbf{A}'&\varepsilon \mathbf{B}\\
\varepsilon\mathbf{C} &\varepsilon \mathbf{D}
\end{array}
\right)
\end{aligned}
\label{eq:4.1}
\end{equation}
of our main result, theorem \ref{thm:1.2}, with the normalizations \eqref{eq:1.18} -- \eqref{eq:1.20}.
We also recall the notation $\langle \alpha \rangle, \langle \alpha \rangle_h, $ of \eqref{eq:1.21}, \eqref{eq:1.26} for the arithmetic and harmonic means of the diagonal elements $-a\alpha_m$ of the fast $N$-cycle $\mathbf{A}$, with signs $\sigma, \sigma_h = \pm1$.
For the choice
\begin{equation}
0<a \in J:=(\underline{a}, \bar{a})
\label{eq:4.2}
\end{equation}
of the parameter $a$, depending on the feedback sign $\beta = \beta_1= \pm1$, with the remaining off-diagonal elements of $\mathbf{A}$ normalized to $\beta_2=\ldots=\beta_N=1$, see \eqref{eq:1.27}, \eqref{eq:1.28}.

To prove theorem \ref{thm:1.2} we proceed as follows.
First we fix the open subset $\mathcal{U}\subseteq J \times \mathbb{R}^M$, where we seek global Hopf bifurcation, according to definition \ref{def:2.2} and corollary \ref{cor:2.4}.
In lemma \ref{lem:4.1}, we then check the crucial assumption \eqref{eq:2.11}, i.e.
\begin{equation}
\sum_\mathcal{P} \zhong \neq 0,
\label{eq:4.3}
\end{equation}
for the center indices $\zhong$ of the Hopf points in $\mathcal{U}$, at $\varepsilon=0.$
An elementary perturbation argument will extend \ref{eq:4.3} to small enough $0< \varepsilon <\varepsilon_0$, proving the theorem.

Fix $\varepsilon> 0$ small enough.
Let $(a,x)\in \mathcal{E} \subseteq J \times \mathbb{R}^M$ denote the steady states $f(\varepsilon,a,x)=0$, and distinguish the trivial steady state $f(\varepsilon,a,x^*)=0$ from the complementary ones:
\begin{equation}
\mathcal{E}^*:= J \times\{x^*\}, \qquad \mathcal{E}^c:=\mathcal{E}\setminus\mathcal{E}^*.
\label{eq:4.4}
\end{equation}
Eliminating all nontrivial steady states $\mathcal{E}^c$ from further consideration, we define
\begin{equation}
\mathcal{U}:= (J \times \mathbb{R}^M)\setminus \mathcal{E}^c
\label{eq:4.5}
\end{equation}
as the open background set for global Hopf bifurcation.
In other words, the trivial line $\mathcal{E}^*$ is the set of steady states in $\mathcal{U}$, and 
\begin{equation}
\label{eq:4.4a}
\mathcal{H}:= \{(a_n, x^*)\in\mathcal{E}^* \, | \,(a_n, x^*)\ \textrm{is a Hopf point of}\ f(\varepsilon, a_n, \cdot)\}
\end{equation}
is the set of Hopf points in $\mathcal{U}$.
Note that $\mathcal{H}$ is finite, by analyticity of the linearization $\mathcal{A}=f_x(\varepsilon,a,x^*)$ in $a$.

\begin{lem}
Let $\varepsilon=0$.
Then the number of Hopf points $(a_n,x^*) \in \mathcal{E}^*$ for the fast subsystem $\dot{\xi}= \mathbf{A}(a)\xi$ is finite, and
\begin{equation}
\sum_n \zhong (a_n, x^*)\neq0,
\label{eq:4.6}
\end{equation}
under any of the assumptions \eqref{eq:1.22}--\eqref{eq:1.25}.
\label{lem:4.1}
\end{lem}
\begin{proof}[\textbf{Proof.}]
By transverse crossings of eigenvalues, in proposition \ref{eq:3.2}, we may consider $\underline{a}=0$ in \eqref{eq:1.27}, and $\bar{a}=1,\infty$ in \eqref{eq:1.28}, for $\varepsilon=0$, without loss of generality.
Let us consider the case $(-1)^{N_{\mathrm{aut}}}= -\beta,\ \bar{a}= \infty$, first, where proposition \ref{prop:3.2} asserts absence of zero eigenvalues of $\mathbf{A}(a)$, for all $a\geq 0$.
Then
\begin{equation}
\zhong(a_n,x^*)=(-1)^{\mu(a_n)} \chi(a_n)
\label{eq:4.7}
\end{equation}
all share the same $n$-independent prefactor
\begin{equation}
(-1)^{\mu(a_n)}= (-1)^{\mu(\infty)}= (-1)^{N_{\mathrm{aut}}}= -\beta .
\label{eq:4.8}
\end{equation}
Indeed all strictly unstable dimensions $\mu(a)$ share the same even/odd parity, by absence of any zero eigenvalue.
Moreover, the local crossing numbers $\chi(a_n)$ at $\underline{a}=0<a_n<\bar{a}=\infty$ just add up to a net crossing number $\chi:=\sum \chi(a_n)$, and therefore
\begin{equation}
2 \sum_n \zhong (a_n,x^*)= -2\beta\cdot\chi= -\beta \cdot(\mu(\infty)-\mu(0^+)).
\label{eq:4.9}
\end{equation}
It remains to show $\mu(\infty)-\mu(0^+) \neq 0$.

Several subcases arise.
Consider the case $\beta=+1$ first.
Then our assumption $(-1)^{ N_{\mathrm{aut}}}=-\beta$ implies that ${ N_{\mathrm{aut}}}$ is odd.
Suppose $N \not\equiv 0 \ \Mod 4$.
Then substitution of \eqref{eq:3.21} and \eqref{eq:3.18} from proposition \ref{prop:3.3} imply
\begin{equation}
\mu(\infty)- \mu(0^+)= N_{\mathrm{aut}}-2\lfloor N/4\rfloor-1\,.
\label{eq:4.10}
\end{equation}
On the other hand, the present case $\beta=+1,\ N \not\equiv 0 \ \Mod 4$ allows us to invoke the two-line assumption \eqref{eq:1.22}.
Because $N_{\mathrm{aut}}$ is odd, however, the first line of that assumption containing $\sigma_h=\pm 1$ cannot hold.
The second line of assumption \eqref{eq:1.22}, which does not contain $\sigma_h$, therefore asserts $\mu(\infty)-\mu(0^+) \neq 0$ in \eqref{eq:4.10}, as claimed above.

For $\beta=+1,\ N \equiv 0 \ \Mod4$, we similarly substitute \eqref{eq:3.22} and \eqref{eq:3.18} from proposition \ref{prop:3.3}.
Invoking the second line of assumption \eqref{eq:1.23}, without $\sigma_h$, then implies
\begin{equation}
\mu(\infty)- \mu(0^+)= N_{\mathrm{aut}}-2 N/4+ \sigma \neq 0.
\label{eq:4.11}
\end{equation}

The cases arising from $\beta = -1$, where ${ N_{\mathrm{aut}}}$ is even, are treated analogously, invoking the second lines of assumptions \eqref{eq:1.24}, \eqref{eq:1.25} without $\sigma_h$, this time.

It remains to consider the casuistics of $(-1)^{ N_{\mathrm{aut}}}= +\beta$.
This time, a simple eigenvalue $\lambda_k=0$ appears at $a=\bar{a}=1$.
Proposition \ref{prop:3.3} guarantees absence of Hopf points, for $a\geq1$, i.e. $\mu(1 +)=\mu (\infty)$.
Proposition \ref{prop:3.2} asserts the transverse crossing direction $\mathrm{sign}\, \lambda_k'(\alpha)= -\sigma_h \neq 0$, at $a=1$, i.e. $\mu(1^+)- \mu(1^-)= -\sigma_h$.
Together, this shows
\begin{equation}
\begin{aligned}
2\ \sum_n\zhong (a_n,&x^*)\ =\ -2 \beta\cdot \chi= -\beta \cdot(\mu(1^-) - \mu(0^+))= \\
\ &=\ -\beta\cdot(\mu(\infty)- (\mu(\infty)- \mu(1^+))- (\mu(1^+)- \mu(1^-))- \mu(0^+))\\
\ &=\ \beta\cdot (N_{\mathrm{aut}} - 0 + \sigma_h - \mu(0^+)).
\end{aligned}
\label{eq:4.12}
\end{equation}
Due diligence analogous to \eqref{eq:4.10} -- \eqref{eq:4.12}, but invoking the first lines of our assumptions \eqref{eq:1.22}--\eqref{eq:1.25} which contain $\sigma_h$\,, this time, completes the proof of the lemma.
\end{proof}
\begin{proof}[\textbf{Proof of theorem \ref{thm:1.2}}]
Let $(a_n,x^*)$ enumerate the finitely many Hopf points of $\mathbf{A}(a)$, at $\varepsilon=0$, ordered such that $0<a_1<a_2\ldots<a_{\bar{n}}$.
Recall that $a=0$ is a Hopf point if, and only if, $N \equiv 1-\beta \ \Mod 4$, by proposition \ref{prop:3.2}.
Fix any $0<\underline{a}<a_1$, in that case, and $\underline{a}=0$, otherwise; see \eqref{eq:1.27}.
Similarly, $\lambda_k(a)=0$ occurs, for any $a\geq0$, if, and only if, $(-1)^{N_{\mathrm{aut}}}= \beta$.
By proposition \ref{prop:3.4}, we then have $a_{\bar{n}}<1$ and we may fix any $a_{\bar{n}}<\bar{a}<1$, in that case, and $\bar{a}= + \infty$ otherwise; see \eqref{eq:1.28}.
To prove theorem \ref{thm:1.2} we invoke corollary \ref{cor:2.4}.
In the setting \eqref{eq:4.1} -- \eqref{eq:4.5}, it is therefore our only remaining task to show 
\begin{equation}
\sum_{n=1}^{\bar{n}(\varepsilon)} \,\zhong (a_n(\varepsilon),x^*) \neq 0,
\label{eq:4.13}
\end{equation}
for small enough $0<\varepsilon<\varepsilon_0$, and for all perturbed Hopf points $(a_n(\varepsilon), x^*)$ of the perturbed matrix family $\mathcal{A}(\varepsilon, a)$ in \eqref{eq:4.1}. Again, $a_n(\varepsilon) \in J= (\underline{a},\bar{a})$ are ordered such that
\begin{equation}
\underline{a}<a_1(\varepsilon)<a_2(\varepsilon)<\ldots<a_{\bar{n}(\varepsilon)}< \bar{a}\,.
\label{eq:4.14}
\end{equation}
At $\varepsilon=0$, and for any $a \in \mathbb{R}$, the matrix $\mathcal{A}$ is block diagonal, with upper left block $\mathbf{A}(a)$, upper right block  $\varepsilon \mathbf{B}=0$, lower left block $\varepsilon \mathbf{C}=0$, and lower right block $\varepsilon \mathbf{D}=0$.
Standard perturbation theory then asserts $\mathrm{spec}\ \mathcal{A}$ to be given by two disjoint components:
\begin{equation}
\mathrm{spec}\ \mathcal{A}(\varepsilon,a)= (\mathrm{spec}\ \mathbf{A}(a)+ o(1))\ \dot{\cup} \  \varepsilon (\mathrm{spec}\ \mathbf{D}+ o(1));
\label{eq:4.15}
\end{equation}
see for example \cite{Kato80}, section II.6.
Uniformity of the spectral splitting for $a\rightarrow \infty$ follows from the diagonal limit of $a^{-1} \mathcal{A}(\varepsilon, a)$.
Disjointness, for $\varepsilon_0$ small enough and uniformly for $a \in J$, follows from proposition \ref{prop:3.1}, and the excision of the only zero eigenvalue of $\mathbf{A}$ at $a=1$ in case $(-1)^{N_{\mathrm{aut}}}= \beta$.
By our hyperbolicity assumption on $\mathbf{D}$, in theorem \ref{thm:1.2}, only the perturbed part $\mathrm{spec}\,\mathcal{A}+o(1)$ contributes any Hopf points to the sum \eqref{eq:4.13}, and eigenvalues $\lambda_k=0$ remain excluded, for $a \in J = (\underline{a},\bar{a})$ and $0<\varepsilon<\varepsilon_0$.
Note, however, that the specific finite number $\bar{n}(\varepsilon)$ of Hopf points may fluctuate, due to conceivably  nontransverse crossings of the Hopf eigenvalues of $\mathbf{A}$ through the imaginary axis, for some $a \in J$ at $\varepsilon=0$.
Nevertheless
\begin{equation}
\zhong (a_n(\varepsilon),x^*)=-\beta\, (-1)^{M-N}\, \mathrm{sign} \,\det \mathbf{D}\cdot \chi(a_n(\varepsilon),x^*)
\label{eq:4.16}
\end{equation}
allows summation of the crossing numbers $\chi(a_n(\varepsilon),x^*)$, over $n$, to a net crossing number $\chi$, as in \eqref{eq:4.9}, \eqref{eq:4.12}, with
\begin{equation}
2 \sum_n \zhong (a_n(\varepsilon),x^*)=-\beta\, (-1)^{M-N}\, \mathrm{sign} \,\det \mathbf{D} \cdot(\mu(\varepsilon,\bar{a})- \mu(\varepsilon,\underline{a})).
\label{eq:4.17}
\end{equation}
Here the unstable dimensions $\mu$ are evaluated at the fixed boundaries $\underline{a}$ and $\bar{a}$ specified in \eqref{eq:1.27}, \eqref{eq:1.28}, where $\mathcal{A}$ is hyperbolic.
Since \eqref{eq:4.15} implies
\begin{equation}
\mu(\varepsilon,a)=\mu(0,a), \quad \textrm{at}\  \ a= \underline{a}, \bar{a},
\label{eq:4.18}
\end{equation}
lemma \ref{lem:4.1} establishes claim \eqref{eq:4.13}, for small enough $0< \varepsilon<\varepsilon_0$.
This proves our main result, theorem \ref{thm:1.2}.
\end{proof}

We conclude this section with a few comments on the limitations of our result.
Restrictions on the globality trichotomy \eqref{eq:2.15} of the connected component $\mathcal{C}$, in corollary \ref{cor:2.4}, are caused by our domain $\mathcal{U}= (J \times \mathbb{R}^M)\setminus \mathcal{E}^c$; see \eqref{eq:4.5}.
Indeed the second option of \eqref{eq:2.15} calls intersections of $\partial \mathcal{C}$ with 
\begin{equation}
\partial \mathcal{U}= \mathcal{E}^c \cup (\{ \underline{a}, \bar{a}\} \times \mathbb{R}^M)
\label{eq:4.19}
\end{equation}
global.
Here we omit $\bar{a}$ in case $\bar{a}=+\infty$, of course.
Let $(a_n,x_n)$ be a sequence in $\mathcal{C}$ converging to some $(a_\infty,x_\infty) \in \partial \mathcal{U}$, with bounded relevant virtual periods $q_n\rightarrow q_\infty >0$.

Consider the steady state case $(a_\infty, x_\infty)\in \mathcal{E}^c$, first.
Then $(a_\infty, x_\infty)$ is another Hopf point, $x_\infty \neq x^*$, which we had discarded before.
Such a Hopf point may occur on another branch of equilibria, like the nontrivial branch bifurcating from the trivial branch $\mathcal{E}^*$ at $a=1$, in case $(-1)^{N_{\mathrm{aut}}}= \beta$.
Without further assumptions on such nontrivial steady states, this possibility cannot be excluded.

Let us examine the left boundary $a_\infty= \underline{a}$ next.
In case $N \not\equiv 1-\beta\ \Mod 4$, the left endpoint $\underline{a}=0$ of $\mathcal{E}^*$ is not a Hopf point, for any $0< \varepsilon < \varepsilon_0$.
In fact we could have safely extended our analysis into negative scaling coefficients $a$.
The only reason we did not pursue that direction further was our focus on the sign structure of the nonzero diagonal entries $-a\alpha_m$ of $\mathbf{A}$; indeed $N_{\mathrm{aut}}$ counts autocatalytic $\alpha_m<0$ in our analysis.
Reversing all signs of $\alpha_m$, and replacing $N_{\mathrm{aut}}$ by $N-N_{\mathrm{aut}}$, the case of negative $a$ becomes a trivial corollary to time reversal, of course, as discussed at the end of section \ref{sec1b}.

In case $N \equiv 1-\beta\ \Mod 4$, a Hopf point in $\mathcal{E}^*$ of $\mathbf{A}(a)$ occurs at $a=0$, for $\varepsilon=0$.
Without further information on the bifurcation direction of the associated Hopf branch of bifurcating periodic solutions $(a,x)$, we cannot make any assertions concerning the sign of $a$, locally, for small $\varepsilon>0$.
We therefore eliminated this case by fixing a left boundary $a= \underline{a}>0$ for our domain $\mathcal{U}$, in assumption \eqref{eq:1.27}.

Similarly, the right boundary $a=\bar{a}<1$ of \eqref{eq:1.28}, in case $(-1)^{N_{\mathrm{aut}}}= \beta$, eliminated the simple eigenvalue $\lambda_k=0$ of $\mathbf{A}$ at $a=1$ from consideration.
Indeed suppose the simple eigenvalue $\lambda_k (a)$ of order $\langle \alpha \rangle_h \cdot (1-a)$, is of the same order as the perturbation $\varepsilon$.
Then we may consider the resulting interaction with $\varepsilon \mathbf{D},\ldots$ as a rank-1 perturbation of the $(M-N+1) \times (M-N+1)$ block matrix
\begin{equation}
\left(
\begin{array}{cc}
0 &  \\
 & \mathbf{D}
\end{array}
\right).
\label{eq:4.20}
\end{equation}
By pole assignment, this may result in arbitrary spectrum of order $\varepsilon$, including multiple steady state bifurcations and Hopf points.
Simple planar examples $N=1,\ M=2$ illustrate this.
Our choice of $\bar{a}<1$ circumvents such complications.

Finally, our particular choice of the scaling parameter $a$ prevents meaningful results in case $N=2$.
Indeed the resulting matrices
\begin{equation}
\mathbf{A}(a)=
\left(
\begin{array}{cc}
-a\alpha_1 & \beta \\
 1 & -a \alpha_2 
\end{array}
\right),
\label{eq:4.21}
\end{equation}
with $\alpha_2=(-1)^{N_{\mathrm{aut}}}/\alpha_1$, then provide Hopf points $(a,x^*)$ if, and only if, $N_{\mathrm{aut}}=1,\  |\alpha_1|=|\alpha_2 |=1,\ \beta = -1$, and $| a | < 1$.
Such an interval of Hopf points violates our condition that Hopf points be isolated.
We therefore consider $N\geq 3$, only, and leave the planar case to elementary ODE courses.

\section{Four examples}\label{sec5}
We illustrate theorem \ref{thm:1.2} with four examples, in subsections \ref{sec5.1}--\ref{sec5.4} below: the Oregonator model of the Belousov-Zhabotinsky reaction, Volterra-Lotka  population dynamics, the citric acid or Krebs cycle, and a gene regulatory model for mammalian circadian rhythms.
Before we address these specific examples, we recall our basic approach in comparison to existing literature, and comment on some advantages, generalizations, and limitations.

Our result is intended as a quick first test to establish the possibility of sustained autonomous oscillations in a given network.
Many results are available which exclude oscillations, particularly within the setting \eqref{eq:1.5} of mass action kinetics $r_j=k_jx^{y_j}$.
We have already mentioned \cite{HJ74, Mie17, Fei19} above.
Notably, the results in \cite{Fei19} aim to hold for all positive values of the reaction coefficients $k_j$.
The results in \cite{HJ74, Mie17}, in contrast, usually restrict the reaction rate coefficients.
For an example with detailed balance, we recall the Wegscheider relation \eqref{eq:1.9b}.
Complex balance is not a remedy.
Indeed, the assumed existence of a complex balance equilibrium $x^*$ will, in general, impose certain constraints on the reaction rate coefficients: in fact, there are usually more reaction complexes $y_j, \bar{y}_j$ than metabolites $X_m$.

We repeat that it is not our concern here, or below, to run anecdotal numerical simulations for one or the other parameter set of reaction coefficients.
General results on sustained oscillations usually assert the existence of parameters for Hopf points, typically via a Routh-Hurwitz criterion in general dimension $M$.
Even with contemporary methods of computer algebra, and in small dimensions, this remains a formidable task.
See for example \cite{GES05, EEetal15} and the references there.
Transversality and nonresonance conditions for local Hopf bifurcation are usually left unchecked.

On the surface, we generalize these results in at least two ways.
First, our global approach only requires net crossing numbers $\chi$, alias sums of center indices $\zhong$, rather than detailed local analysis.
Second, we allow for quite general reaction rate functions $r_j=r_j(x)$, rather than just mass action kinetics.
That much ``generality", however, comes with a twist.
A third, and quite substantial, generalization to \emph{reversible} fast $N$-cycles with positive feedback arises in the framework of Jacobi systems \eqref{eq:1.9c}, on the basis of proposition \ref{prop:3.1} and \cite{FuOl88}.
We do not pursue that direction further, here.

Let us address the twist of ``generality", which is directed against mass action.
The very setting \eqref{eq:1.16} of a fast $N$-cycle $\mathbf{A}$ in the Jacobian $f_x=(f_{mm'})$ at steady state $x^*$ requires the freedom of a decomposition of the partial derivatives $f_{mm'}$, alias the partials $r_{jm}= \partial_{x_m}r_j(x^*)$, into the fast $N$-cycle $\mathbf{A}$ and the slow remaining partials of order $\varepsilon$, independently from the fixed rates $r_j(x^*)$ themselves which determine the prescribed steady state $x^*$.
Already Michaelis-Menten kinetics \eqref{eq:1.6} provide such freedom of choice:
\begin{equation}
r_{jm}/r_j = \partial_{x_m} \log r_j = \frac{y_{jm}}{x_m^*} \frac{1}{1+c_{jm}x_m^*} \in (0,1)\cdot y_{jm}/x^*_m.
\label{eq:5.1}
\end{equation}
Here we may choose $x_m^*$ as small as we like to guarantee any required range of $r_{jm}$, even for prescribed $r_j(x^*)$.
We thus assumed our choice of $\mathbf{A}$, and the slow-fast decomposition \eqref{eq:1.16} on the linear level, to be independent from $x^*$.
See also the sensitivity analysis \cite{BFie18} and the recent substantial generalizations \cite{Vas20}, which are based on the same concept.
Evidently such independence of $r_{jm}$ from $r_j$ fails in the pure mass action case, where all $c_{jm}=0$.

A second caveat concerns our choice of the distinguished bifurcation parameter $a>0$ in our normalization \eqref{eq:1.20},  $\ a_m=a \alpha_m,\ \prod \alpha_m=\pm1$, of the diagonal entries $a_m=-f_{mm}$ of the fast $N$-cycle $\mathbf{A}$.
Already for N=2, this scaling prevented a meaningful discussion of $2$-cycles, because the necessary $2 \times 2$ Hopf condition $0=\mathrm{tr}\,\mathbf{A}=a(\alpha_1+\alpha_2)$ became invariant under $a$.
For general $N$, for example, consider the presence of an invariant stoichiometric subspace:
\begin{equation}
\mathbf{c}^T \cdot (\bar{y}_j-y_j)=0,
\label{eq:5.2}
\end{equation}
for all $j$, and one or several fixed vectors $\mathbf{c} \neq 0$.
Indeed, \eqref{sec5.2} implies time invariance of any affine hyperplane $\mathbf{c}^Tx=\mathrm{const.}$ under the network ODE \eqref{eq:1.1}.
For the $N$-cycle $\mathbf{A}$, this implies $\mathbf{c}^T\mathbf{A}=0$, i.e.
\begin{equation}
a \mathbf{c}_{m-1}\alpha_{m-1}=\mathbf{c}_m\beta_m,
\label{eq:5.3}
\end{equation}
for all $m \,\Mod N$.
In particular, the characteristic polynomial \eqref{eq:3.8} then reads 
\begin{equation}
0=a^{-N} \mathbf{p}=\prod_{m=1}^N(\lambda/a+ \alpha_m)- \prod_{m=1}^N \alpha_m,
\label{eq:5.4}
\end{equation}
if we assume all $\mathbf{c}_m$ are nonzero on the $N$-cycle $\mathbf{A}$.
Thus all eigenvalues $\lambda$ simply scale radially outward with $a$, from $\lambda=0$.
Such spectral behavior is adverse to Hopf bifurcation.
For this formal reason, for example, we do not treat the replicator equation or Eigen's hypercycle below, which is normalized to the stochastically motivated invariance $x_1+\ldots+x_N=1$.
Of course, we may select other $1$-parameter paths $a_m=a_m(a)$ in such cases, which are more hospitable towards global Hopf bifurcation as in corollary \ref{cor:2.4}.
Or else, we may look for fast $N$-cycles, in the present setting, which are only supported on metabolites $m$ for which $\mathbf{c}_m=0$, if any.
\subsection{Oregonators}\label{sec5.1}
The celebrated standard \emph{Oregonator} \cite{F07} 
is the simplest, chemically somewhat realistic, model of the Belousov-Zhabotinsky oscillatory reaction mechanism; see 
\cite{Zha91, Zha07}.
In our notation \eqref{eq:1.1} the model can be written as
\begin{equation}
\begin{aligned}
\dot{x}_1&=&r_1(x_2)&-r_2(x_1,x_2)&+r_3(x_1)&-r_4(x_1)& \\
\dot{x}_2&=&-r_1(x_2)&-r_2(x_1,x_2)& & &+cr_5(x_3) \\
\dot{x}_3&=& & &2r_3(x_1) & &-r_5(x_3)
\end{aligned}
\label{eq:5.5}
\end{equation}
with mass action rate laws $r_1,\ldots,r_5$ and a stoichiometrically motivated ``fudge factor" $c>0$.
More generally, we admit arbitrary monotone rate laws $r_j$, e.g. of Michaelis-Menten type.
For the linearization $\mathcal{A}=f_x(x^*)$ in \eqref{eq:1.16} we readily obtain
\begin{equation}
\dot{\xi} = 
\left(
\begin{array}{ccc}
r'_3-r'_4-r_{21} \qquad & r'_1-r_{22}&\quad 0\\
-r_{21} \qquad & -r'_1-r_{22}&\quad cr'_5\\
2r'_3 \qquad & 0&\quad -r'_5\\
\end{array}
\right) \xi.
\label{eq:5.6}
\end{equation}
Here we use the abbreviation $r'_j$ for $r_{jm}$, if $r_j=r_j(x_m)$ depends on a single metabolite, only.
The only feasible $N$-cycle involving $N=3=M$ metabolites is $\mathbf{m}=(3 \ 2 \ 1)$,
\begin{equation}
x_2 \xrightarrow{r_1,r_2}x_1 \xrightarrow{r_3}x_3 \xrightarrow{r_5}x_2,
\label{eq:5.7}
\end{equation}
notably with a strongly autocatalytic step $r_3$.
Comparison between \eqref{eq:5.6} and \eqref{eq:1.16} also tells us to consider $\varepsilon:=r_{21}$ as a small perturbation of the $3$-cycle $\mathbf{A}$ in \eqref{eq:1.18}, with the normalizations
\begin{equation}
\begin{aligned}
a\alpha_1&=r'_4-r'_3, \quad &a \alpha_2&=r'_1+r_{22}, \qquad a \alpha_3=r'_5, \\
\beta &= \mathrm{sign}\,(r'_1-r_{22}), \quad &a^3&=(r'_1+r_{22})r'_5 \cdot |r'_4-r'_3|, \\
N_{\mathrm{aut}} &\in \{0,1\}, &\quad (-1)^{N_{\mathrm{aut}}}&= \mathrm{sign}\,\,\alpha_1 = \mathrm{sign} \,(r'_4-r'_3).\\
\end{aligned}
\label{eq:5.8}
\end{equation}
Since $N=3 \not\equiv 0,2 \ \Mod4$, theorem \ref{thm:1.2} asserts global Hopf bifurcation as follows.
If $\beta= +1$, i.e. for $r'_1> r_{22}$ at steady state $x^*$, we are in case \eqref{eq:1.22}.
Therefore $N_{\mathrm{aut}} \in \{0,1\}$ requires $N_{\mathrm{aut}}=1+\sigma_h$, which implies $N_{\mathrm{aut}}= 0$ and $\sigma_h=-1$.
These contradictory requirements exclude the case $\beta=+1$ of a positive feedback cycle \eqref{eq:5.7}, where $r'_1$ dominates.

For $r'_1<r_{22}$, i.e. for a negative feedback cycle $\beta=-1$, in contrast, the restrictions \eqref{eq:1.24} are satisfied if, and only if $N_{\mathrm{aut}}=0$ or $N_{\mathrm{aut}}=1=-\sigma_h$. 
Specifically this leads to the two cases
\begin{equation}
\begin{aligned}
&r'_1 < r_{22} \quad \mathrm{and} \quad 0 > r'_3-r'_4 \,, \qquad \mathrm{or\  else} \\
&r'_1 < r_{22} \quad \mathrm{and} \quad 0 < r'_3-r'_4 < \left( \tfrac{1}{r'_1+r_{22}}+ \tfrac{1}{r'_5} \right)^{-1}.
\end{aligned}
\label{eq:5.9}
\end{equation}
In conclusion, \eqref{eq:5.9} implies global Hopf bifurcation for the generalized Oregonator with $a \in (0, \bar{a})$, any small $1-\bar{a}>0$, and for $\varepsilon:=r_{21}<\varepsilon_0(\bar{a})$ small enough.
\subsection{Lotka-Volterra networks}\label{sec5.2}
In the introduction we have mentioned the planar Lotka system \cite{Lot1920} for oscillating chemical reactions.
Going far beyond that classical ``predator-prey" system, Volterra \cite{Vol1931} first studied quadratic systems of the general form
\begin{equation}
\dot{x}_m=x_m(c_m+ \sum_{m'}a_{mm'}x_{m'}),
\label{eq:5.10}
\end{equation}
with $m=1,\ldots,M, \ x_m>0$, in the context of ecological population dynamics.
See \cite{Oli14} for an excellent survey.
Usually $a_{mm'},a_{m'm}>0$ indicate mutually beneficial \emph{cooperation} or \emph{symbiosis} between different species $m$ and $m'$, whereas $a_{mm'},a_{m'm}<0$ model mutually toxic \emph{competition}.
 The \emph{predator-prey} case of Lotka is $a_{mm'} \cdot a_{m'm}<0$.
Following standard ecological wisdom, we assume self-inhibition $a_{mm}<0$ to prevent unlimited grow-up or blow-up of single species.

We may rescale any fixed positive equilibrium $x^*$ to $x^*_m=1$, for all $m$, without loss of generality.
Then the linearization of \eqref{eq:5.10} at $x^*$ is given by
\begin{equation}
\dot{\xi}=(a_{mm'})\,\xi.
\label{eq:5.11}
\end{equation}
In particular we may examine any (relabeled) feedback $N$-cycle $\mathbf{m} = (1\ldots N)$,
\begin{equation}
\beta= \mathrm{sign} \, (a_{12}\cdot\ldots\cdot a_{N-1,N} \cdot a_{N1}) = \pm 1
\label{eq:5.12}
\end{equation}
with normalized diagonal
\begin{equation}
a\alpha_m:= -a_{mm}>0,
\label{eq:5.13}
\end{equation}
i.e. $N_{\mathrm{aut}}=0$. 
In particular note $\sigma= \mathrm{sign} \, \langle \alpha \rangle= \mathrm{sign} \, \langle 1 / \alpha \rangle = \sigma_h= +1$.

For positive feedback $N$-cycles $\beta=+1$, conditions \eqref{eq:1.22}, \eqref{eq:1.23} boil down to
\begin{equation}
N\geqslant 5, \mathrm{\ in\ case \ } \beta =+1.
\label{eq:5.14}
\end{equation}
Indeed, suppose $N \not\equiv 0 \ \Mod 4$.
Then $N_{\mathrm{aut}}=0$ in \eqref{eq:1.22} and $\sigma_h=+1$ require $2\lfloor N / 4\rfloor+1 >1$, i.e. $N\geq 5$, for $N \not \equiv 0 \ \Mod 4$.
For $N \equiv 0 \ \Mod 4$, condition \eqref{eq:1.23} and $\sigma= \sigma_h=+1$ similarly requires $N /4=0$ or $|2N/ 4-1|>1$, i.e. $N / 4 \geq 2$.
This proves claim \eqref{eq:5.14}.

For negative feedback $N$-cycles $\beta=-1$, conditions \eqref{eq:1.24}, \eqref{eq:1.25} analogously require 
\begin{equation}
N\geqslant 3,  \mathrm{ \ in \ case}\ \beta =-1.
\label{eq:5.15}
\end{equation}

It remains to specify the parameter region $a \in (\underline{a}, \bar{a})$ of global Hopf bifurcation, according to \eqref{eq:1.27}, \eqref{eq:1.28}.
For case \eqref{eq:5.14} we obtain the conditions
\begin{equation}
\begin{array}{rcl}
\beta &=& +1, \quad \ N\geqslant5; \\
\underline{a}&:=&0 \quad  \mathrm{ for} \ N \not \equiv 0 \ \Mod 4, \quad \mathrm{else}\ \underline{a}>0;\\
\bar{a}&<&1;
\end{array}
\label{eq:5.15a}
\end{equation}
since $N_{\mathrm{aut}}=0$.
Similarly, case \eqref{eq:5.15} summarizes as
\begin{equation}
\begin{array}{rcl}
\beta &=& -1, \quad \ N\geqslant3; \\
\underline{a}&:=&0 \quad  \mathrm {for} \ N \not \equiv 2 \ \Mod 4, \quad \mathrm{else} \ \underline{a}>0;\\
\bar{a}&:=& \infty.
\end{array}
\label{eq:5.15b}
\end{equation}
Assuming other interactions $a_{mm'}$ to be of sufficiently small order $\varepsilon$, and hyperbolicity of the diagonal block $\varepsilon \mathbf{D}$ complementary to the $N$-cycle $\mathbf{m} = (1\ldots N)$, theorem \ref{thm:1.2} implies global Hopf bifurcation for parameters $a \in (\underline{a},\bar{a})$ as described in \eqref{eq:5.15a}, \eqref{eq:5.15b}.


\subsection{Citric acid cycles}\label{sec5.3}
The \emph{citric acid cycle} (CAC) or \emph{Krebs cycle} is a central hub of the oxidative energy metabolism in any cell; see for example \cite{BTGS15}, chapter 17.
Although variants depend on taxonomy, the following 8-cycle of enzymatic Michaelis-Menten reactions is a central feature:%
\begin{equation}
\begin{aligned}
r_m: \ X_m \rightarrow X_{m+1}, \quad m \,\Mod 8\,.
\end{aligned}
\label{eq:5.16}
\end{equation}
Here $X_1 =$ Citrate, $X_2=$ Isocitrate, $X_3=\alpha$-Ketoglutarate, $X_4=$ Succinyl-coenzyme A,
$X_5=$ Succinate, $X_6=$ Fumarate, $X_7=$ Malate, and $X_8=$ Oxaloacetate.
Side reactions and regulatory influences are omitted.
Oscillations have been observed, experimentally, in mitochondria extracts of liver and pancreatic cells; see \cite{MacDetal03}.
One motivation is to understand oscillations in insulin production.

In absence of self-regulation, the fast monomolecular feedback $8$-cycle \eqref{eq:5.16} with rates $r_m=r_m(x_m)$ does not provide global Hopf bifurcation.
Indeed, linearization of \eqref{eq:5.16} at a steady state $x^*$ provides the fast cycle $\mathbf{A}$ in \eqref{eq:1.18}, with $a_m=r'_m= \beta_{m+1}$.
This determines the scaling parameter $a$ to be fixed at $a=1$; see \eqref{eq:1.20}.
Moreover $\lambda_0=0$ is the eigenvalue with maximal real part, by its positive (left) eigenvector and for positive feedback $\beta=+1$; see propositions \ref{prop:3.1}, \eqref{prop:3.2}(i).
In particular, the steady state $x^*$ is linearly stable and Hopf bifurcation is excluded.

Regulatory and self-regulatory controls of the CAC \eqref{eq:5.16}, however, are biologically essential.
Otherwise energy conversion would run high, for no reason and with nowhere to go.
In our setting, regulatory feedbacks are the primary focus.
Consider an arbitrary $M$-cycle \eqref{eq:5.16}, with $m \,\Mod M$.
Assume, however, that metabolite $X_N$ up- or down-regulates reaction $r_0:\ X_0 \rightarrow X_1$, enzymatically, i.e. without any appreciable effect on the mass balance of $X_N$ itself.
In other words,
\begin{equation}
r_0=r_0(x_0,x_N)\,,
\label{eq:5.17}
\end{equation}
and $r_m=r_m(x_m)$ remains monomolecular for all $m \neq 0$.
Deviating from the stereotypical monotonicity assumption \eqref{eq:1.7a} on the partial derivatives $r_{jm}$ we have to admit $r_{0N}<0$ here, to account for the observed inhibitory regulation of reaction $r_0$ by metabolite $X_N$.
Most importantly, we assume the partial derivative $r_{00}$ to be small of order $\varepsilon$, along with all other partial derivatives outside the fast $N$-cycle $\mathbf{m}=(1 \ldots N)$ defined by $r'_1,\ldots,r'_N$, and $r_{0N}$.
This provides a fast $N$-cycle matrix $\mathbf{A}$, as in \eqref{eq:1.18}, given by
\begin{equation}
a_m=r'_m>0, \quad \beta_m=r'_{m-1}>0, \quad \mathrm{except \ for\ } \beta_1= r_{0N}<0,
\label{eq:5.18}
\end{equation}
with $m=1,\ldots,N$.
Normalization yields
\begin{equation}
\alpha_m>0, \ N_{\mathrm{aut}}=0, \quad  \beta=\mathrm{sign} \ r_{0N}=-1, \quad \mathrm{and} \ a^N= |r'_N / r_{0N}|.
\label{eq:5.19}
\end{equation}
In particular our previous Lotka-Volterra discussion of section \ref{sec5.2} applies, with negative feedback $\beta=-1$.
We obtain global Hopf bifurcation for any cycle length $N\geq 3$; see \eqref{eq:5.15}.

Let us return to the CAC case \eqref{eq:5.16} of the experiments in \cite{MacDetal03}.
The following four regulatory terms are listed there:
\begin{equation}
\begin{aligned}
r_8(x_8,x_1), & \quad \mathrm{with} &\quad &N=1;\\
r_8(x_8,x_4), & \quad \mathrm{with} &\quad &N=4;\\
r_3(x_3,x_4), & \quad \mathrm{with} &\quad &N=1;\\
r_5(x_5,x_8), & \quad \mathrm{with} &\quad &N=3.
\end{aligned}
\label{eq:5.20}
\end{equation}
All these regulations act by enzyme inhibition, i.e. $\beta=-1$.
The only exception, excitatory self-regulation $\tilde{r}_3(x_3)$ of $x_3$ by the input $x_3$ itself, can be subsumed into the definition of the rate $r_3$ and is therefore omitted.
Regulation with $N=1$ has to be considered small, in our setting, because it is anticipatory, on the wrong side of the diagonal of $\mathbf{A}$, along the 8-cycle \eqref{eq:5.18}.
Therefore \eqref{eq:5.15} only provides global Hopf bifurcation by the two inhibitory feedbacks $r_8(x_8,x_4)$ and $r_5(x_5,x_8)$, separately, on $a \in (0, \infty)$ and under suitable smallness and nondegeneracy conditions for the large number of remaining entries in a full model of the CAC metabolism.

In summary, our result for $N=4$ in \eqref{eq:5.20} points at the inhibitory effect of $X_4=$ Succinyl-coenzyme A
 on the Citrate synthase reaction $r_8$, which produces $X_1=$ Citrate from $X_8=$ Oxaloacetate, as a possible regulatory source of the observed oscillations.
 
Similarly, $N=3$ in \eqref{eq:5.20} points at the inhibitory effect of $X_8=$ Oxaloacetate on the Succinate dehydrogenase reaction $r_5$ from $X_5=$ Succinate to $X_6=$ Fumarate, as a second possible regulatory cause of oscillations.

\subsection{Circadian gene regulation}\label{sec5.4}
Gene regulatory mechanisms for \emph{circadian rhythms}, on the cell level, have received considerable attention over the past decades; see \cite{Lal17} for some references concerning \emph{drosophila}.
A gene regulatory model for cells in the suprachiasmatic nucleus of mammals was developed, among others, by \cite{Miretal09}; see \cite{MFieMKS13} for simulations of periodic orbits in that model.
The model involves a total of $M=21$ components, with $8$ gene activities transcripted intro mRNAs, $8$ corresponding proteins, and $5$ heterodimers of proteins.
Below we write gene activities in small italics, and proteins in capitals.
Except for dimerizations, all reactions are of Michaelis-Menten type \eqref{eq:1.6}, with numerous enzyme inhibitory feedback cycles.
To indicate such inhibition of a reaction $j:\ X_m \rightarrow X_{\bar{m}}$ by a metabolite $X_{m'}$ we write $X_{m'} \dashv (X_m \rightarrow X_{\bar{m}})$.
Gene transcription in itself does not lower gene activity; rather we may consider such steps $j: \ y_j\rightarrow \bar{y}_j$ as autocatalytic, $\bar{y}_{jm}= y_{jm} \neq 0$, in the language of \eqref{eq:1.3}, without depleting $x_m$.
We indicate such steps by arrows $\mapsto$.
Finally, all components are subject to linear decay rates.

We do not bother here to write down the complete model network, the ODE model, or any of the more than $150$ rate coefficients, many of them guesswork anyway.
Instead we highlight the following cycles:
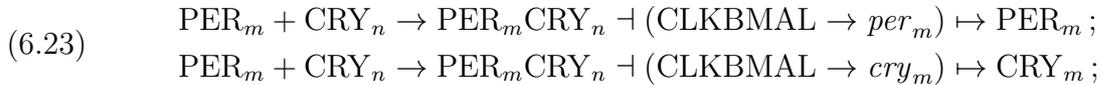
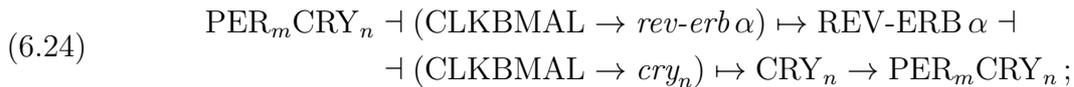
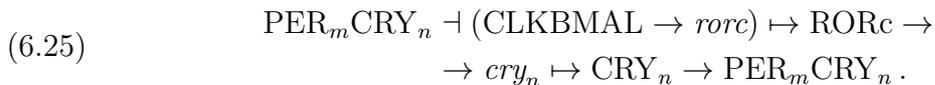
\begin{equation}
\begin{aligned}
\mathrm{PER}_m + \mathrm{CRY}_n  \rightarrow  \mathrm{PER}_m \mathrm{CRY}_n & \dashv(\mathrm{CLKBMAL} \rightarrow \mathit{per}_m)\mapsto \mathrm{PER}_m \, ;\\
\mathrm{PER}_m + \mathrm{CRY}_n  \rightarrow  \mathrm{PER}_m \mathrm{CRY}_n & \dashv(\mathrm{CLKBMAL} \rightarrow \mathit{cry}_m)\mapsto \mathrm{CRY}_m \, ;\\
\end{aligned}
\label{eq:5.21}
\end{equation}
\begin{equation}
\begin{aligned}
\mathrm{PER}_m \mathrm{CRY}_n  &\dashv (\mathrm{CLKBMAL} \rightarrow \emph{rev-erb}\,\alpha) \mapsto  \mathrm{REV{\hbox{-}}ERB\,\alpha} \dashv\\
      &\dashv (\mathrm{CLKBMAL} \rightarrow \emph{cry}_n) \mapsto \mathrm{CRY}_n \rightarrow \mathrm{PER}_m\mathrm{CRY}_n\,;
\end{aligned}
\label{eq:5.22}
\end{equation}
\begin{equation}
\begin{aligned}
\mathrm{PER}_m \mathrm{CRY}_n  &\dashv (\mathrm{CLKBMAL} \rightarrow \emph{rorc}) \mapsto  \mathrm{RORc} 
\rightarrow \\
                                                          &\rightarrow \emph{cry}_n \mapsto \mathrm{CRY}_n \rightarrow \mathrm{PER}_m\mathrm{CRY}_n\,.
\end{aligned}
\label{eq:5.23}
\end{equation}
Here $m,n \in \{1,2\}$ distinguish two variants of the $\emph{per, cry}$ genes and PER, CRY proteins.
Since all components $x_m$ are subject to decay, and in absence of strong autocatalysis, we obtain corresponding fast $N$-cycles $\mathbf{A}$ with positive $\alpha_m$, in \eqref{eq:1.18}.
In particular $N_{\mathrm{aut}}=0, \ \sigma= \sigma_h=1$, with arbitrary $a>0$, and $\beta = \pm1$ as follows:
\begin{equation}
\begin{aligned}
N=3,\ \beta=-1 \quad \mathrm{in} \ \eqref{eq:5.21};\\
N=5,\ \beta=+1 \quad \mathrm{in} \ \eqref{eq:5.22};\\
N=5,\ \beta=-1 \quad \mathrm{in} \ \eqref{eq:5.23}.\\
\end{aligned}
\label{eq:5.24}
\end{equation}
For the positive fast feedback cycle $N=5$ in \eqref{eq:5.22}, our analysis \eqref{eq:5.15a} in the Lotka-Volterra section \ref{sec5.2} asserts global Hopf bifurcation for scaling parameters $a \in (0,\bar{a})$ and any $\bar{a}<1$, under the usual smallness and hyperbolic nondegeneracy conditions.
For the negative fast feedback cycles $N=3,5$ in \eqref{eq:5.21}, \eqref{eq:5.23}, in contrast, our analysis \eqref{eq:5.15b} in the same section \ref{sec5.2} asserts global Hopf bifurcation for the whole interval of scaling parameters $a \in (0, \infty)$.


\end{document}